\documentclass[reqno,10pt]{amsart}

\usepackage{amsmath, amsthm, amsfonts}
\usepackage{pdfsync}
\usepackage{parskip}
\usepackage{hyperref}
\usepackage{mathrsfs}  
\usepackage{mathtools}
\usepackage{mdframed}
\usepackage{tikz}
\numberwithin{equation}{section}
\newtheorem{theorem}{Theorem}[section]
\newtheorem{prop}[theorem]{Proposition}
\newtheorem{definition}[theorem]{Definition}
\newtheorem{lem}[theorem]{Lemma}
\newtheorem{cor}[theorem]{Corollary}
\newtheorem{example}[theorem]{Example}
\theoremstyle{remark}
\newtheorem{remark}[theorem]{Remark}

\def\M{\mathsf{M}}
\def\sM{\mathscr{M}}
\def\p{\mathsf{p}}
\def\d{{\sf d}}

\def\ev{\mathsf{C}}
\def\cH{\mathcal{H}}
\def\sH{\mathscr{H}}
\def\cE{\mathcal{E}}
\def\fE{\frak{E}}
\def\cP{\mathcal{P}_\lambda}

\def\J{\mathbb{J}}

\def\R{\mathbb{R}}
\def\C{\mathbb{C}}
\def\H{\mathbb{H}^{m}}

\def\N{\mathbb{N}}
\def\Nz{\mathbb{N}_0}
\def\A{\mathfrak{A}}
\def\K{\frak{K}}

\def\L{\mathcal{L}}
\def\Lis{\mathcal{L}{\rm{is}}}
\def\F{\mathfrak{F}}
\def\bF{\boldsymbol{\mathfrak{F}}}
\def\bW{\boldsymbol{W}}
\def\bu{\boldsymbol{u}}

\def\B{\mathbb{B}}

\def\Ok{\mathsf{O}_{\kappa}}

\def\vpk{\varphi_{\kappa}}
\def\psk{\psi_{\kappa}}

\def\Q{\mathsf{Q}^{m}}
\def\Qk{\mathsf{Q}^{m}_{\kappa}}

\def\cA{\mathcal{A}}
\def\sA{\mathscr{A}}
\def\Ah{\mathcal{A}_h}
\def\Aw{\mathcal{A}_\omega}
\def\Awla{\mathcal{A}_{\omega}^*(\lambda^\prime)}
\def\Awl{\mathcal{A}_{\omega}(\lambda^\prime)}
\def\Ata{\mathcal{A}^*(\vartheta)}
\def\Azt{\mathcal{A}(\vartheta;2\alpha)}
\def\Aa{\mathcal{A}^*}
\def\Aath{\mathcal{A}^*_h(\vartheta)}
\def\cB{\mathcal{B}}
\def\fB{\mathfrak{B}}

\def\Bw{\mathcal{B}_\omega}

\def\aw{\mathfrak{a}_\omega}

\def\ba{\boldsymbol{\vec{a}}}

\def\kb{\varphi^{\ast}_{\kappa}}
\def\kf{\psi^{\ast}_{\kappa}}

\def\pk{\pi_\kappa}

\def\Re{\mathcal{R}}
\def\Rek{\mathcal{R}_{\kappa}}
\def\Rc{\mathcal{R}^c}
\def\Rck{\mathcal{R}^c_{\kappa}}

\def\Rp{{\rm{Re}}}
\def\id{{\rm{id}}}
\def\supp{{\rm{supp}}}
\def\gd{{\rm{grad}}}

\def\div{{\rm div}}

\def\sg{{\rm{sign}}}
\def\Hom{{\text{Hom}}}

\def\X{\mathbb{X}}
\def\Xk{\mathbb{X}_{\kappa}}

\begin{document}

\title[Singular parabolic equations on singular manifolds]{Singular parabolic equations of second order on manifolds with singularities}

\author[Y. Shao]{Yuanzhen Shao}
\address{Department of Mathematics,
         Vanderbilt University, 
         1326 Stevenson Center, 
         Nashville, TN 37240, USA}
\email{yuanzhen.shao@vanderbilt.edu}


\begin{abstract}
The main aim of this article is to establish an $L_p$-theory for elliptic operators on manifolds with singularities.
The particular class of differential operators discussed herein may exhibit degenerate or singular behavior near the singular ends of the manifolds. Such a theory is of importance for the study of elliptic and parabolic equations on non-compact, or even incomplete manifolds, with or without boundary.
\end{abstract}
\maketitle

\section{\bf Introduction}

In this article, we study second order differential operators in an $L_p$-framework defined on manifolds with singularities. The particular class of manifolds considered here is called {\em singular manifolds}. Roughly speaking, a manifold $(\M,g)$ is singular if it is conformal to a manifold $(\M,g/\rho^2)$ whose local patches are of comparable sizes, and all transit maps and curvatures are uniformly bounded. The conformal factor $\rho$ is called a singularity function for $(\M,g)$. In \cite{DisShaoSim}, it is shown that the class of all such $(\M,g/\rho^2)$ coincides with the family of complete manifolds with bounded geometry if we restrict ourselves to manifolds without boundary. The concept of {\em singular manifolds} used in this paper is first introduced by H. Amann in \cite{Ama13}.

The approach in this article is based on the traditional strategy of associating differential operators with densely defined, closed and sectorial forms. 
This method, being utilized by many authors, has displayed its power in establishing $L_p$-semigroup theory for second order differential operators on domains in $\R^N$. See, for example, \cite{ArenElst97, ArenThom05, Dav89, MouSel12, Ouh92, Ouh04} and the references therein. 
To clarify the role of the differential operators in this article, we look at 
\begin{align}
\label{S1: weighted diff op}
\sA=\rho^{-\lambda} \cA ,
\end{align}
where $\rho\in C^\infty(\M, (0,1))$ is a conformal factor and $\lambda>0$, or $\rho\in C^\infty(\M, (1,\infty))$   and $\lambda<0$. 
$\cA$ is a uniformly strongly $\rho$-elliptic operator in the sense that the local expressions of $\cA$ have uniform ellipticity constants in all local coordinates. More precisely, a second order differential operator
$$\cA u:= -\div(\vec{a}\cdot\gd u)+\ev(\nabla u, a_1) +a_0 u$$
is  uniformly strongly $\rho$-elliptic, if the principal symbol of $\cA$ fulfils
$$ \hat{\sigma}\cA(x,\xi):= (\vec{a}(x)\cdot\xi | \xi)_{g^*} \sim \rho^2 |\xi|_{g^*}^2, $$
for any cotangent field $\xi$. 
Here $g^*$ is the cotangent metric induced by $g$, $\vec{a}$ is a symmetric $(1,1)$-tensor field on $(\M,g)$, and the operation $[u\mapsto\vec{a}\cdot\nabla u]$ denotes center contraction. See Section~3 for the precise definition.
The $L_p$ theory of uniformly strongly $\rho$-elliptic operators has been established by H. Amann in \cite{Ama13b}. 

In contrast, in this paper  we will focus on the operator $\sA$. An easy computation shows that the principal symbol of $\sA$ satisfies
$$\hat{\sigma}\sA(x,\xi)  \sim \rho^{2-\lambda} |\xi|_{g^*}^2.$$
Therefore, $\sA$ can exhibit both degenerate and singular behaviors near the singular ends.
However, in comparison to  $\cA$, 
the choice of $\rho$ and $\lambda$ in \eqref{S1: weighted diff op} reveals that  
the ellipticity constants of the localizations for the operator $\sA$  in local coordinates blow up while approaching the singular ends of the manifold $(\M,g)$.
The rate of the blow-up for the ellipticity constant is characterized by the power $\lambda$. 
For this reason, we will call such $\sA$ a $(\rho,\lambda)$-singular elliptic operator.
The precise definition of  $(\rho,\lambda)$-singular ellipticity can be found in Section~3.
To illustrate the behavior of the operator $\sA$, we consider the Euclidean space $\R^N$ as a {\em singular manifold} with $\infty$ as a singular end, and take $\cA$ to be the Laplacian in \eqref{S1: weighted diff op}. 
Then the operator $\sA$, in some sense, looks like one with unbounded coefficients at infinity on $\R^N$.

To the best of the author's knowledge, there are only very few papers on the generation of analytic semigroups for  differential operators with unbounded diffusion coefficients in  $\R^N$ or in an exterior domain with regular boundary, among them \cite{ForLor07, GiuGozMonVes08, GozMonVes02,  MetPalPruSch05, MouSel12}. 
In all  these articles, the drift coefficients have to be controlled by the diffusion and potential terms. 
In \cite{MouSel12}, the authors use a form operator method to prove a semigroup result for operators with unbounded coefficients in a weighted Sobolev space. 
The drawback of the method used in \cite{MouSel12} is reflected by the difficulty to precisely determine the domains of the differential operators. This is, in fact, one of the most challenging tasks in the form operator approach. 
One of the most important features of this article is that with the assistance of the theory for function spaces and differential operators on {\em singular manifolds} established in \cite{Ama13, AmaAr,  Ama13b}, we can find a precise characterization for the domains of the second order  $(\rho,\lambda)$-singular elliptic operators.

A conventional method to render the associated sesquilinear form of an elliptic operator $\cA$  densely defined, closed and sectorial is to perturb $\cA$ by a spectral parameter $\omega>0$. 
See \cite{MouSel12, Ouh04} for instance. 
Then $\cA$ generates a quasi-contractive semigroup.
However, for a $(\rho,\lambda)$-singular elliptic operator, e.g., the operator $\sA$ in \eqref{S1: weighted diff op}, because of the existence of the multiplier $\rho^{-\lambda}$,  we need to perturb $\sA$ by a weight function of the form $\omega\rho^{-\lambda}$.
This feature arising from our approach creates an essential difficulty for parabolic theory of differential equation on manifolds with singularities. 
We take conical manifolds as an example. Given a compact closed manifold $B$, the Laplacian on the conical manifold $([0,1)\times B) / (\{0\}\times B)$  reads as 
$$t^{-2}((t \partial_t)^2 + \Delta_B). $$
In order to prove that this operator generates a contractive semigroup, we need to perturb it not by a constant $\omega$, but actually by a weight function $\omega t^{-2}$.
The commutator of weight functions and differential operators is usually not a perturbation in the sense of \cite{Fuhr97, PruSim07}. 
Thus the extra term $\omega t^{-2}$, in general, cannot be removed by a ``soft" method, like the perturbation theory of semigroups.
In some cases, e.g., the Laplacian operator, we find it more practical to put a control on the diffusion or drift term. 
This is a quite natural condition which has been used in \cite{ForLor07, GiuGozMonVes08, GozMonVes02,  MetPalPruSch05, MouSel12}. 
In all  these articles, the growth of the drift coefficients have to be controlled by the diffusion and potential terms.

In Section~5, we are able to remove the compensation condition $\omega \rho^{-\lambda}$  for a class of {\em singular manifolds} called {\em singular manifolds} with  $\sH_\lambda$-ends. 
To the best of the author's knowledge, this concept is introduced here for the first time.  
To illustrate how to construct such manifolds, we look at the following example of manifolds with ``holes". First, we start with an $m$-dimensional complete closed manifold $(\mathscr{M},g)$ with bounded geometry. Then we remove finitely many $\Sigma_j\subset \mathscr{M}$. Each $\Sigma_j$ is an $m$-dimensional compact manifold with boundary. Let 
$$\M:=\mathscr{M}\setminus \cup_{j}\Sigma_j.$$ 
Since the boundary $\partial\Sigma_j$ is not contained in $\M$, the manifold $(\M,g)$ is incomplete. The resulting manifold with ``holes" is a {\em singular manifold} with  $\sH_\lambda$-ends. 

To illustrate the work in this paper, we consider the Laplace-Beltrami operator 
$$\Delta_g:=\div_g\circ \gd_g$$ 
on a manifold with ``holes", which we denote by $(\M,g)$.
We want to point out that  $-\Delta_g$ is indeed of the same type as $\sA$ in \eqref{S1: weighted diff op}. Instead, taking 
$$\rho:={\sf dist}(\cdot,\partial\Sigma_j),\quad \text{near }\Sigma_j;\quad \rho\sim {\bf 1}\quad \text{elsewhere}, $$ 
the operator $-\rho^2 \Delta_g$ is uniformly strongly $\rho$-elliptic. 
Here $\sim$ denotes Lipschitz equivalence.
In Section~5.3, we prove that $\Delta_g$ generates a strongly continuous analytic semigroup on $L^{\lambda^\prime}_p(\M)$ with domain $W^{2,\lambda^\prime-2}_p(\M)$ for any $1<p<\infty$. Here $L^{\lambda^\prime}_p(\M)$ and $W^{2,\lambda^\prime-2}_p(\M)$
are some weighted Sobolev spaces whose definition will be given in Section 2.2.  More general results for second order differential operators will be stated in Section~5.3 below.

The study of differential operators on manifolds with singularities is motivated by a variety of applications from applied mathematics, geometry and topology. 
All of it is related to the seminal paper
by V.A. Kondrat'ev \cite{Kon67}.
There is a tremendous amount of literature on pseudo-differential calculus of differential operators of Fuchs type, which have been introduced independently by R.B.~Melrose \cite{Mel81, Mel93} and B.-W.~Schulze \cite{NazSalSchSte06, Sch94, Sch97, SchSei02}. One branch of these lines of research is connected with the so-called $b$-calculus and its generalizations on manifolds with cylindrical ends. See \cite{Mel81, Mel93}. Many authors have been very active in this direction. Research along another line, known as conical differential operators, has also been known for a long time. 
Operators in this line of research are modelled on conical manifolds. The investigation of conical singularities was initiated by J.~Cheeger in \cite{Chg79, Chg80, Chg83}, and then continued by many other authors. A comparison between the $b$-calculus and the cone algebra can be found in \cite{LauSei99}. However, for higher order singularities, the corresponding algebra becomes far from being elementary, although many ideas and structures can be extracted, e.g., from the calculus of boundary value problems, c.f., \cite{LiuWit04, Sch94, SchSei02}. In Section~5.2, we will show that it is possible to create {\em singular manifolds} with $\sH_\lambda$-ends with singularities of arbitrarily high dimension.  The amount of research on pseudo-differential calculus of differential operators of Fuchs type is enormous, and thus it is literally impossible to list all the work. 

This paper is organized as follows.

In the next section, we present some preliminary material, including the definitions and fundamental properties of the function spaces used in this article, and a divergence theorem for tensor bundles.

Section 3 provides the theoretical basis for this paper, wherein we prove the generation of analytic $L_p$-semigroups by second order differential operators in divergence form on {\em singular manifolds}. 
To prove that a differential operator $\sA$ generates a contractive strongly continuous analytic semigroup, as we mentioned earlier, it is usually necessary to perturb $\sA$ by a weight function of the form $\omega\rho^{-\lambda}$, which is equivalent to requiring $\sA$ to possess a large positive potential term. 
A precise bound on this compensation condition can be formulated for $L_2$-theory, or general $L_p$-theory for scalar functions.
It is shown in Section~5 that, for {\em singular manifolds} with  $\sH_\lambda$-ends, the aforementioned largeness condition for the potential term of $\sA$, or equivalently the perturbation $\omega\rho^{-\lambda}$, can be removed.

In Section 4, we follow the techniques and constructions from \cite{Ama14} to introduce two important classes of {\em singular manifolds}, that is, manifolds with singularities of wedge type and manifolds with holes. Typical examples of manifolds of wedge type  are conical manifolds and edge manifolds. 
As mentioned in Remark~\ref{S5.2: remove bdry} below,  the concept of manifolds with holes can be generalized to manifolds constructed by removing finitely many compact closed submanifolds  from a complete manifold.

In Section~5, we first demonstrate a technique to remove the compensation condition on the potential terms  formulated in Section~3  for second order differential operators defined on {\em singular manifolds} with so called {\em property $\sH_\lambda$}, which means that there exists some function $h\in C^2(\M)$ on $(\M,g)$ with singularity function $\rho$ satisfying
$$\rho|\gd_g h|_g\sim {\bf 1}_\M,\quad \rho^\lambda \div_g(\rho^{2-\lambda}\gd_g h)\sim  {\bf 1}_\M.$$
The study of such conditions is new.
Based on this technique, we generalize the $L_p$-theory established in Section~3 to a class of manifolds, called {\em singular manifolds} with $\sH_\lambda$-ends. 
Roughly speaking, a {\em singular manifold} has $\sH_\lambda$-ends if near the singularities it is a {\em singular manifold} with  {\em property $\sH_\lambda$}.  
The discussions in Section~5.2 and 5.3 show how to construct {\em singular manifolds} with $\sH_\lambda$-ends in a systematic way.
The main results of this articles, Theorems~\ref{S5: sing mfd-ends}, \ref{S5: main theorem} and Corollary~\ref{S5: main corollary}, are presented in Section~5.3.

In the last section, several applications of the $L_p$-theory established in Section~5 are given.
First, we apply the theory established in Section~5 to the heat equation on {\em singular manifolds} with $\sH_\lambda$-ends to establish an existence and uniqueness result in an $L_p$-framework. 
The second example
concerns  parabolic equations with lower order degeneracy or boundary singularity on domains with compact boundary.
The order of the degeneracy or singularity is measured by the rate of decay or blow-up in the ellipticity constant while approaching the boundary.
This example generalizes the results in \cite{ForMetPall11, Shao14, Vesp89}.
In the third example, we discuss a generalization of the parabolic Heston equation.
One feature of the equations considered in the second and third examples is the anisotropic degeneracy of the higher order and lower order terms.
For instance, while the leading term is degenerate towards the boundary, the lower order terms are allowed to exhibit boundary singularities.

{\textbf {Assumptions on manifolds:}}
Following H.~Amann \cite{Ama13, AmaAr},
let $(\M,g)$ be a $C^\infty$-Riemannian manifold of dimension $m$ with or without boundary endowed with $g$ as its Riemannian metric such that its underlying topological space is separable. An atlas $\A:=(\Ok,\vpk)_{\kappa\in \K}$ for $\M$ is said to be normalized if 
\begin{align*}
\vpk(\Ok)=
\begin{cases}
\Q, \hspace*{1em}& \Ok\subset\mathring{\M},\\
\Q\cap\H, &\Ok\cap\partial\M \neq\emptyset,
\end{cases}
\end{align*}
where $\H$ is the closed half space $\bar{\R}^+ \times\R^{m-1}$ and $\Q$ is the unit cube at the origin in $\R^m$. We put $\Qk:=\vpk(\Ok)$ and  $\psk:=\vpk^{-1}$. 
\smallskip\\
The atlas $\A$ is said to have \emph{finite multiplicity} if there exists $K\in \N $ such that any intersection of more than $K$ coordinate patches is empty. Put
\begin{align*}
\mathfrak{N}(\kappa):=\{\tilde{\kappa}\in\K:\mathsf{O}_{\tilde{\kappa}}\cap\Ok\neq\emptyset \}.
\end{align*} 
The finite multiplicity of $\A$ and the separability of $\M$ imply that $\A$ is countable.
\smallskip\\
An atlas $\A$ is said to fulfil the \emph{uniformly shrinkable} condition, if it is normalized and there exists $r\in (0,1)$ such that $\{\psk(r{\Qk}):\kappa\in\K\}$ is a cover for ${\M}$. 
\smallskip\\
Following H.~Amann \cite{Ama13, AmaAr}, we say that $(\M,g)$ is a {\bf{uniformly regular Riemannian manifold}} if it admits an atlas $\A$ such that
\begin{itemize}
\item[(R1)] $\A$ is uniformly shrinkable and has finite multiplicity. If $\M$ is oriented, then $\A$ is orientation preserving.
\item[(R2)] $\|\varphi_{\eta}\circ\psk \|_{k,\infty}\leq c(k) $, $\kappa\in\K$, $\eta\in\mathfrak{N}(\kappa)$, and $k\in{\N}_0$.
\item[(R3)] $\kf g\sim g_m $, $\kappa\in\K$. Here $g_m$ denotes the Euclidean metric on ${\R}^m$ and $\kf g$ denotes the pull-back metric of $g$ by $\psk$.
\item[(R4)] $\|\kf g\|_{k,\infty}\leq c(k)$, $\kappa\in\K$ and $k\in\Nz$.
\end{itemize}
Here $\|u\|_{k,\infty}:=\max_{|\alpha|\leq k}\|\partial^{\alpha}u\|_{\infty}$, and it is understood that a constant $c(k)$, like in (R2), depends only on $k$. An atlas $\A$ satisfying (R1) and (R2) is called a \emph{uniformly regular atlas}. (R3) reads as
\begin{center}
$|\xi|^2/c\leq \kf g(x)(\xi,\xi) \leq{c|\xi|^2}$,\hspace{.5em} for any $x\in \Qk,\xi\in \R^m, \kappa\in\K$ and some $c\geq{1}$.
\end{center}
In \cite{DisShaoSim}, it is shown that the class of {\em uniformly regular Riemannian manifolds} coincides with the family of complete Riemannian manifolds with bounded geometry, when $\partial\M=\emptyset$.

Assume that $\rho\in C^{\infty}(\M,(0,\infty))$. Then $(\rho,\K)$ is a {\em singularity datum} for $\M$ if
\begin{itemize}
\item[(S1)] $(\M,g/\rho^2)$ is a {\em uniformly regular Riemannian manifold}.
\item[(S2)] $\A$ is a uniformly regular atlas.
\item[(S3)] $\|\kf\rho\|_{k,\infty}\leq c(k)\rho_{\kappa}$, $\kappa\in\K$ and $k\in\N_0$, where $\rho_{\kappa}:=\rho(\psk(0))$.
\item[(S4)] $\rho_{\kappa}/c\leq \rho(\p)\leq c\rho_{\kappa}$, $\p\in\Ok$ and $\kappa\in\K$ for some $c\geq 1$ independent of $\kappa$.
\end{itemize}
Two {\em singularity data} $(\rho,\K)$ and $(\tilde{\rho},\tilde{\K})$ are equivalent, if
\begin{itemize}
\item[(E1)] $\rho\sim \tilde{\rho}$.
\item[(E2)] card$\{\tilde{\kappa}\in\tilde{\K}:\mathsf{O}_{\tilde{\kappa}}\cap\Ok\neq\emptyset\}\leq c$, $\kappa\in\K$.
\item[(E3)] $\|\varphi_{\tilde{\kappa}}\circ\psk\|_{k,\infty}\leq{c(k)}$, $\kappa\in\K$, $\tilde{\kappa}\in\tilde{\K}$ and $k\in{\N}_0$
\end{itemize}
We write the equivalence relationship as $(\rho,\K)\sim(\tilde{\rho},\tilde{\K})$. (S1) and (E1) imply that 
\begin{align}
\label{section 1:singular data}
1/c\leq \rho_{\kappa}/\tilde{\rho}_{\tilde{\kappa}}\leq c,\hspace*{.5em} \kappa\in\K,\hspace*{.5em} \tilde{\kappa}\in\tilde{\K}\text{ and }\mathsf{O}_{\tilde{\kappa}}\cap\Ok\neq \emptyset.
\end{align}
{\em A singularity structure}, $\mathfrak{S}(\M)$, for $\M$ is a maximal family of equivalent {\em singularity data}. A {\em singularity function} for $\mathfrak{S}(\M)$ is a function $\rho\in C^{\infty}(\M,(0,\infty))$ such that there exists an atlas $\A$ with $(\rho,\A)\in\mathfrak{S}(\M)$. The set of all {\em singularity functions} for $\mathfrak{S}(\M)$ is the {\em singular type}, $\mathfrak{T}(\M)$, for $\mathfrak{S}(\M)$. By a {\bf{singular manifold}} we mean a Riemannian manifold $\M$ endowed with a singularity structure $\mathfrak{S}(\M)$. Then $\M$ is said to be \emph{singular of type} $\mathfrak{T}(\M)$. If $\rho\in\mathfrak{T}(\M)$, then it is convenient to set $[\![\rho]\!]:=\mathfrak{T}(\M)$ and to say that $(\M,g;\rho)$ is a {\em singular manifold}. A {\em singular manifold} is a {\em uniformly regular Riemannian manifold} iff $\rho\sim {\bf 1}_{\M}$. 
\smallskip\\
We refer to \cite{Ama13b, Ama14} for examples of {\em uniformly regular Riemannian manifolds} and {\em singular manifolds}.
\smallskip\\
A {\em singular manifold} $\M$ with a uniformly regular atlas $\A$ admits a \emph{localization system subordinate to} $\A$, by which we mean a family $(\pi_{\kappa} )_{\kappa\in\K}$ satisfying:
\begin{itemize}
\item[(L1)] $\pk \in\mathcal{D}(\Ok,[0,1])$ and $(\pi_{\kappa}^{2})_{\kappa\in \K} $ is a partition of unity subordinate to $\A$.
\item[(L2)] $\|\kf  \pk \|_{k,\infty} \leq c(k) $, for $\kappa\in\K$, $k\in \Nz$.
\end{itemize}
The reader may refer to \cite[Lemma~3.2]{Ama13} for a proof. 

Lastly, for each $k\in\N$, the concept of {\bf{$C^k$-uniformly regular Riemannian manifold}} is defined by modifying (R2), (R4) and (L1), (L2) in an obvious way. Similarly, {\bf{$C^k$-singular manifolds}} are defined by replacing the smoothness of $\rho$ by $\rho\in C^k(\M,(0,\infty))$ and altering (S1)-(S3) accordingly.

\textbf{Notations:} 
Given any topological set $U$, $\mathring{U}$ denotes the interior of $U$. 
\smallskip\\
For any two Banach spaces $X,Y$, $X\doteq Y$ means that they are equal in the sense of equivalent norms. The notation $\Lis(X,Y)$ stands for the set of all bounded linear isomorphisms from $X$ to $Y$.
\smallskip\\
Given any Banach space $X$ and  manifold $\mathscr{M}$,
let $\| \cdot \|_\infty$, $\| \cdot \|_{s,\infty}$, $\|\cdot \|_p$ and $\|\cdot \|_{s,p}$ denote the usual norm of the Banach spaces $BC(\mathscr{M},X)$($L_\infty(\mathscr{M},X)$), $BC^s(\mathscr{M},X)$, $L_p(\mathscr{M},X)$ and $W^s_p(\mathscr{M},X)$, respectively. 
\smallskip\\ 

\section{\bf Preliminaries}

In this Section, we follow the work of H. Amann in \cite{Ama13} and \cite{AmaAr} to introduce some concepts and properties of weighted function spaces on {\em singular manifolds}.
Let $\mathbb{A}$ be a countable index set. Suppose $E_{\alpha}$ is for each $\alpha\in\mathbb{A}$ a locally convex space. We endow $\prod_{\alpha}E_{\alpha}$ with the product topology, that is, the coarsest topology for which all projections $pr_{\beta}:\prod_{\alpha}E_{\alpha}\rightarrow{E_{\beta}},(e_{\alpha})_{\alpha}\mapsto{e_{\beta}}$ are continuous. By $\bigoplus_{\alpha}E_{\alpha}$ we mean the vector subspace of $\prod_{\alpha}E_{\alpha}$ consisting of all finitely supported elements, equipped with the inductive limit topology, that is, the finest locally convex topology for which all injections $E_{\beta}\rightarrow\bigoplus_{\alpha}E_{\alpha}$ are continuous. 

\subsection{\bf Tensor bundles}
Suppose $(\M,g;\rho)$ is a {\em singular manifold}.
Given $\sigma,\tau\in\N_0$, 
$$T^{\sigma}_{\tau}{\M}:=T{\M}^{\otimes{\sigma}}\otimes{T^{\ast}{\M}^{\otimes{\tau}}}$$ 
is the $(\sigma,\tau)$-tensor bundle of $\M$, where $T{\M}$ and $T^{\ast}{\M}$ are the tangent and the cotangent bundle of ${\M}$, respectively.
We write $\mathcal{T}^{\sigma}_{\tau}{\M}$ for the $C^{\infty}({\M})$-module of all smooth sections of $T^{\sigma}_{\tau}\M$,
and $\Gamma(\M,T^{\sigma}_{\tau}{\M})$ for the set of all sections.

For abbreviation, we set $\J^{\sigma}:=\{1,2,\ldots,m\}^{\sigma}$, and $\J^{\tau}$ is defined alike. Given local coordinates $\varphi=\{x^1,\ldots,x^m\}$, $(i):=(i_1,\ldots,i_{\sigma})\in\J^{\sigma}$ and $(j):=(j_1,\ldots,j_{\tau})\in\J^{\tau}$, we set
\begin{align*}
\frac{\partial}{\partial{x}^{(i)}}:=\frac{\partial}{\partial{x^{i_1}}}\otimes\cdots\otimes\frac{\partial}{\partial{x^{i_{\sigma}}}}, \hspace*{.5em} \partial_{(i)}:=\partial_{i_{1}}\circ\cdots\circ\partial_{i_{\sigma}} \hspace*{.5em} dx^{(j)}:=dx^{j_1}\otimes{\cdots}\otimes{dx}^{j_{\tau}}
\end{align*}
with $\partial_{i}=\frac{\partial}{\partial{x^i}}$. The local representation of 
$a\in \Gamma(\M,T^{\sigma}_{\tau}{\M})$ with respect to these coordinates is given by 
\begin{align}
\label{local}
a=a^{(i)}_{(j)} \frac{\partial}{\partial{x}^{(i)}} \otimes dx^{(j)} 
\end{align}
with coefficients $a^{(i)}_{(j)}$ defined on $\Ok$。
\smallskip\\
We denote by $\nabla=\nabla_g$ the Levi-Civita connection on $T{\M}$. It has a unique extension over $\mathcal{T}^{\sigma}_{\tau}{\M}$ satisfying, for $X\in\mathcal{T}^1_0{\M}$,
\begin{itemize}
\item[(i)] $\nabla_{X}f=\langle{df,X}\rangle$, \hspace{1em}$f\in{C^{\infty}({\M})}$,
\item[(ii)] $\nabla_{X}(a\otimes{b})=\nabla_{X}a\otimes{b}+a\otimes{\nabla_{X}b}$, \hspace{1em}$a\in\mathcal{T}^{{\sigma}_1}_{{\tau}_1}{\M}$, $b\in\mathcal{T}^{{\sigma}_2}_{{\tau}_2}{\M}$,
\item[(iii)] $\nabla_{X}\langle{a,b}\rangle=\langle{\nabla_{X}a,b}\rangle+\langle{a,\nabla_{X}b}\rangle$, \hspace{1em}$a\in\mathcal{T}^{{\sigma}}_{{\tau}}{\M}$, $b\in\mathcal{T}^{{\tau}}_{{\sigma}}{\M}$,
\end{itemize}
where $\langle{\cdot,\cdot}\rangle:\mathcal{T}^{\sigma}_{\tau}{\M}\times{\mathcal{T}^{\tau}_{\sigma}{\M}}\rightarrow{C^{\infty}({\M})}$ is the extension of the fiber-wise defined duality pairing on ${\M}$, cf. \cite[Section 3]{Ama13}. Then the covariant (Levi-Civita) derivative is the linear map
$$\nabla: \mathcal{T}^{\sigma}_{\tau}{\M}\rightarrow{\mathcal{T}^{\sigma}_{\tau+1}{\M}},\quad a\mapsto{\nabla{a}}$$
defined by
$$\langle{\nabla{a},b\otimes{X}}\rangle:=\langle{\nabla_{X}a,b}\rangle,\quad b\in\mathcal{T}^{\tau}_{\sigma}{\M},\quad X\in\mathcal{T}^{1}_{0}{\M}.$$
For $k\in{\N}_0$, we define
$$\nabla^k: \mathcal{T}^{\sigma}_{\tau}{\M}\rightarrow{\mathcal{T}^{\sigma}_{\tau+k}{\M}},\quad a\mapsto{\nabla^k{a}}$$
by letting $\nabla^0 a:=a$ and $\nabla^{k+1} a:=\nabla\circ\nabla^k a$.
We can also extend the Riemannian metric $(\cdot|\cdot)_g$ from the tangent bundle to any $(\sigma,\tau)$-tensor bundle $T^{\sigma}_{\tau}{\M}$ such that $(\cdot|\cdot)_g:=(\cdot|\cdot)_{g^\tau_\sigma}:T^{\sigma}_{\tau}{\M}\times{T^{\sigma}_{\tau}{\M}}\rightarrow \C $ by 
$$(a|b)_g =g_{(i)(\tilde{i})} g^{(j)(\tilde{j})}	a^{(i)}_{(j)}  \bar{b}^{(\tilde{i})}_{(\tilde{j})}$$
in every coordinate with $(i),(\tilde{i})\in \mathbb{J}^\sigma$, $(j),(\tilde{j})\in \mathbb{J}^\tau$ and 
$$ g_{(i)(\tilde{i})}:= g_{i_1,\tilde{i}_1}\cdots g_{i_\sigma,\tilde{i}_\sigma},\quad g^{(j)(\tilde{j})}:= g^{j_1,\tilde{j}_1}\cdots g^{j_\tau,\tilde{j}_\tau}.$$ 
In addition,
$$|\cdot|_g:=|\cdot|_{g^\tau_\sigma}:\mathcal{T}^{\sigma}_{\tau}{\M}\rightarrow{C^{\infty}}({\M}),\quad a\mapsto\sqrt{(a|a)_g}$$
is called the (vector bundle) \emph{norm} induced by $g$.
\smallskip\\
We assume that $V$ is a $\C$-valued tensor bundle on $\M$ and $E$ is a $\C$-valued vector space, i.e.,
\begin{center}
$V=V^{\sigma}_{\tau}:=\{T^{\sigma}_{\tau}\M, (\cdot|\cdot)_g\}$,\hspace{.5em} and\hspace{.5em} $E=E^{\sigma}_{\tau}:=\{\C^{m^{\sigma}\times m^{\tau}},(\cdot|\cdot)\}$, 
\end{center}
for some $\sigma,\tau\in\N_0$. Here $(a|b):=$trace$(b^{\ast}a)$ with $b^{\ast}$ being the conjugate matrix of $b$. By setting $N=m^{\sigma+\tau}$ , we can identify $\F^s(\M,E)$ with $\F^s(\M)^N$. 

Recall that for any $a\in V^\sigma_{\tau+1}$
$$(a^\sharp)^{(i;k)}_{(j)}:=g^{kl} a^{(i)}_{(j;l)}, \quad (i)\in\J^\sigma,\quad(j)\in\J^\tau,\quad  k,l\in\J^1. $$
We have 
$|a^\sharp|_{g^\tau_{\sigma+1}}=|a|_{g^{\tau+1}_\sigma}.$
For any $(i_1)\in \mathbb{J}^{\sigma_1}$ and $(i_2)\in \mathbb{J}^{\sigma_2}$, the index $(i_1;i_2)$ is defined by
\begin{align*}
(i_1;i_2)=(i_{1,1},\cdots,i_{1,\sigma_1};i_{2,1},\cdots,i_{2,\sigma_2}).
\end{align*} 
Given any $a\in V^{\sigma+1}_\tau$,
$$(a_\flat)^{(i)}_{(j;k)}:=g_{kl} a^{(i;l)}_{(j)}.$$
Similarly, we have
$|a_\flat|_{g^{\tau+1}_\sigma}=|a|_{g_{\sigma+1}}^\tau. $
\smallskip\\
Suppose that $\sigma+\tau\geq 1$. We put for $a\in V$ and $\alpha_i\in T^*\M$, $\beta^j\in T\M$
$$(G^\tau_\sigma a)(\alpha_1,\cdots,\alpha_\tau; \beta^1,\cdots,\beta^\sigma):=a((\beta^1)_\flat ,\cdots,(\beta^\sigma)_\flat;(\alpha_1)^\sharp,\cdots,(\alpha_\tau)^\sharp ).$$
Then it induces a conjugate linear bijection
$$G^\tau_\sigma: V\rightarrow V^\prime, \quad (G^\tau_\sigma)^{-1}=G^\sigma_\tau. $$
Consequently, for $a,b\in V$
$$(a|b)_g= \langle a, G^\tau_\sigma \bar{b}\rangle. $$
From this, it is easy to show
\begin{align}
\label{S2: G^tau_sigma}
|G^\tau_\sigma a|_{g^\sigma_\tau}=|a|_{g^\tau_\sigma}.
\end{align}
Throughout the rest of this paper, unless stated otherwise, we always assume that 
\smallskip
\begin{mdframed}
\begin{itemize}
\item $(\M,g;\rho)$ is a {\em singular manifold}.
\item $\rho\in \mathfrak{T}(\M)$, $s\geq 0$, $1<p<\infty$ and $\vartheta\in \R$.
\item $(\pk,\zeta_{\kappa})_{\kappa\in\K}$ is a localization system subordinate to $\mathfrak{A}$.
\item $\sigma,\tau\in \Nz$, $V=V^{\sigma}_{\tau}:=\{T^{\sigma}_{\tau}\M, (\cdot|\cdot)_g\}$, $E=E^{\sigma}_{\tau}:=\{\C^{m^{\sigma}\times m^{\tau}},(\cdot|\cdot)\}$.
\end{itemize}
\end{mdframed}
In \cite[Lemma~3.1]{Ama13}, it is shown that $\M$ satisfies the following properties:
\begin{itemize}
\item[(P1)] ${\kf}g\sim \rho^2_{\kappa}g_m$ and ${\kf}g^{\ast}\sim \rho^{-2}_{\kappa}g_m$, where $g^{\ast}$ is the induced contravariant metric.
\item[(P2)] $\rho^{-2}_{\kappa}\|{\kf}g\|_{k,\infty}+\rho^2_{\kappa}\|{\kf}g^{\ast}\|_{k,\infty}\leq c(k)$, $k\in\N_0$ and $\kappa\in\K$.
\item[(P3)] For $\sigma,\tau\in\N_0$ given, then 
\begin{center}
${\kf}(|a|_g)\sim \rho^{\sigma-\tau}_{\kappa}|{\kf}a|_{g_m}$, \hspace{1em}$a\in \mathcal{T}^{\sigma}_{\tau}\M$,
\end{center}
and
\begin{center}
$|{\kb}b|_g \sim \rho^{\sigma-\tau}_\kappa \kb(|b|_{g_m})$, \hspace{1em}$b\in \mathcal{T}^{\sigma}_{\tau}\Qk$.
\end{center}
\end{itemize}
\smallskip
For $K\subset \M$, we put $\K_{K}:=\{\kappa\in \K: \Ok\cap K\neq\emptyset\}$. Then, given $\kappa\in\K$,
\begin{align*}
\Xk:=
\begin{cases}
\R^m \hspace*{1em}\text{if }\kappa\in \K\setminus \K_{\partial\M},\\
\H \hspace*{1em}\text{otherwise,}
\end{cases}
\end{align*}
endowed with the Euclidean metric $g_m$.

Given $a\in \Gamma(\M,V)$ with local representation $\eqref{local}$
we define ${\kf}a\in E$ by means of
$ {\kf}a=[a^{(i)}_{(j)}]$,
where $[a^{(i)}_{(j)}]$ stands for the $(m^{\sigma}\times m^{\tau})$-matrix with entries $a^{(i)}_{(j)}$ in the $((i),(j))$ position, with $(i)$, $(j)$ arranged lexicographically.

\subsection{\bf Weighted function spaces}
For the sake of brevity, we set $\boldsymbol{L}_{1,loc}(\X,E):=\prod_{\kappa}{L}_{1,loc}({\Xk},E)$. Then we introduce two linear maps for $\kappa\in\K$:
\begin{center}
$\Rck:{L}_{1,loc}({\M},V)\rightarrow{L}_{1,loc}({\Xk},E)$, $u\mapsto{\psi_{\kappa}^{\ast}({\pk}u)}$,
\end{center}
and
\begin{center}
$\Rek:{L}_{1,loc}({\Xk},E)\rightarrow{L}_{1,loc}({\M},V)$, $v_{\kappa}\mapsto{\pk}{\kb}v_{\kappa}$.
\end{center}
Here and in the following it is understood that a partially defined and compactly supported tensor field is automatically extended over the whole base manifold by identifying it to be zero  outside its original domain.
We define
$$\Rc:{L}_{1,loc}({\M},V)\rightarrow\boldsymbol{L}_{1,loc}(\R^m),\quad u\mapsto{(\Rck u)_{\kappa}},$$
and
$$\Re:\boldsymbol{L}_{1,loc}(\R^m)\rightarrow{L}_{1,loc}({\M},V),\quad (v_{\kappa})_{\kappa}\mapsto{\sum_{\kappa}\Rek v_\kappa}.$$ 
\smallskip\\
In the rest of this subsection we suppose that  $k\in\Nz$. 
We denote by $\mathcal{D}({\M},V)$ the space of smooth sections of $V$ that is compactly supported in $\M$.
Then the weighted Sobolev space $W^{k,\vartheta}_p({\M},V)$ is defined as the completion of $\mathcal{D}({\M},V)$ in $L_{1,loc}(\M,V)$ with respect to the norm
\begin{center}
$\|\cdot\|_{k,p;\vartheta}: u\mapsto(\sum_{i=0}^{k}\|\rho^{\vartheta+i+\tau-\sigma}|\nabla^{i}u|_{g}\|_p^p)^{\frac{1}{p}}$.
\end{center}
Note that $W^{0,\vartheta}_p(\M, V)=L^{\vartheta}_p(\M, V)$ with equal norms. In particular, we can define the weighted spaces $L_q^\vartheta(\M,V)$ for $q\in \{1,\infty\}$ in a similar manner.

Analogously, the weighted Besov spaces are defined for $k\in\N$ by
\begin{align}
\label{S2: Besov space}
B^{k,\vartheta}_p(\M,V):=(W^{k-1,\vartheta}_p(\M, V),W^{k+1,\vartheta}_p(\M, V))_{1/2,p}. 
\end{align}
Define
$$BC^{k,\vartheta}(\M,V):=(\{u\in{C^k({\M},V)}:\|u\|_{k,\infty;\vartheta}<\infty\},\|\cdot\|_{k,\infty;\vartheta}),$$
where $\|u\|_{k,\infty;\vartheta}:={\max}_{0\leq{i}\leq{k}}\|\rho^{\vartheta+i+\tau-\sigma}|\nabla^{i}u|_{g}\|_{\infty}$.
We also set
$$BC^{\infty,\vartheta}({\M},V):=\bigcap_{k}BC^{k,\vartheta}(\M, V).$$ 
The weighted Sobolev-Slobodeckii spaces are defined as
\begin{align}
\label{S2: fractional Sobolev}
W^{s,\vartheta}_p({\M},V):=(L^{\vartheta}_p(\M, V),W^{k,\vartheta}_p(\M, V))_{s/k,p},
\end{align}
for $ s\in \R_+\setminus\Nz$, $k=[s]+1$,
where $(\cdot,\cdot)_{\theta,p}$ is the real interpolation method \cite[Section~1.3]{Trib78}.
\smallskip\\
Whenever $\partial\M\neq \emptyset$,
we denote by 
$\mathring{W}^{s,\vartheta}_p(\M, V)$ the closure of $\mathcal{D}(\mathring{\M},V)$ in $W^{s,\vartheta}_p(\M, V)$.
In particular,
$$\mathring{W}^{s,\vartheta}_p(\M, V)=W^{s,\vartheta}_p(\M, V),\quad 0\leq s <1/p.$$ 
See \cite[Theorem~8.3(ii)]{Ama13}.

In the special case that $(\M,g)$ is uniformly regular, since $\rho\sim {\bf 1}_\M$, the definition of any weighted space $\F^{s,\vartheta}(\M,V)$ is actually independent of the weight $\vartheta$. In this case, all spaces are indeed unweighted. We thus denote these spaces simply by $\F^s(\M,V)$.

In the following context, assume that $E_\kappa$ is a sequence of Banach spaces for $\kappa\in\K$. Then $\mathbf{E}:=\prod_{\kappa}E_{\kappa}$. For $1\leq{q}\leq{\infty}$, we denote by $l_q^\vartheta(\mathbf{E}):=l_q^\vartheta(\mathbf{E};\rho)$ the linear subspace of $\mathbf{E}$ consisting of all $\boldsymbol{x}=(x_\kappa)$ such that
\begin{align*}
\|\boldsymbol{x}\|_{l_q^\vartheta(\mathbf{E})}:=
\begin{cases}
(\sum\limits_\kappa\|\rho_\kappa^{\vartheta+m/q } x_\kappa\|^q_{E_\kappa})^{1/q},\quad&1\leq{q}<{\infty},\\
\sup\limits_\kappa\|\rho_\kappa^\vartheta x_\kappa\|_{E_\kappa},& q=\infty
\end{cases}
\end{align*}
is finite. Then $l_q^\vartheta(\mathbf{E})$ is a Banach space with norm $\|\cdot\|_{l^\vartheta_q(\mathbf{E})}$.

For ${\F}\in \{BC,W_p, \mathring{W}_p\}$, we put $\boldsymbol{\F}^s:=\prod_{\kappa}{\F}^s_{\kappa}$, where ${\F}^s_{\kappa}:={\F}^s(\Xk,E)$. 

\begin{prop}
\label{S2: retraction&coretraction}
$\Re$ is a retraction from $l^\vartheta_q(\bF^s)$ onto $\F^{s,\vartheta}(\M, V)$ with $\Rc$ as a coretraction, where $q=p$ for $\F\in \{ W_p, \mathring{W}_p\}$, or $q=\infty$ for $\F=BC$.
\end{prop}
\begin{proof}
See \cite[Theorems~6.1, 6.3, 7.1, 11.1]{Ama13}.
\end{proof}

Let $V_j=V^{\sigma_j}_{\tau_j}:=\{T^{\sigma_j}_{\tau_j}\M,(\cdot|\cdot)_g\}$ with $j=1,2,3$ be $\C$-valued tensor bundles on $\M$. By bundle multiplication from $V_1\times V_2$ into $V_3$, denoted by
\begin{center}
${\mathsf{m}}: V_1\times V_2\rightarrow V_3$,\hspace{1em} $(v_1,v_2)\mapsto {\mathsf{m}}(v_1,v_2)$,
\end{center}
we mean a smooth bounded section $\mathfrak{m}$ of $\Hom(V_1\otimes V_2,V_3)$, i.e., 
\begin{align}
\label{section 2: bundle multiplication}
\mathfrak{m}\in BC^{\infty}(\M, \text{Hom}(V_1\otimes V_2,V_3)), 
\end{align}
such that $\mathsf{m}(v_1,v_2):=\mathfrak{m}(v_1\otimes v_2)$. \eqref{section 2: bundle multiplication} implies that  for some $c>0$
\begin{center}
$|{\mathsf{m}}(v_1,v_2)|_g \leq c|v_1|_g |v_2|_g$,\hspace{1em} $v_i\in \Gamma(\M,V_i)$ with $i=1,2$.
\end{center}
Its point-wise extension from $\Gamma(\M,V_1\oplus V_2)$ into $\Gamma(\M,V_3)$ is defined by:
\begin{align*}
\mathsf{m}(v_1,v_2)(\p):=\mathsf{m}(p)(v_1(\p),v_2(\p))
\end{align*}
for $v_i\in \Gamma(\M,V_i)$ and $p\in\M$. We still denote it by ${\mathsf{m}}$. We can formulate the following point-wise multiplier theorem for function spaces over {\em singular manifolds}.

\begin{prop}
\label{S2: pointwise multiplication}
Let $k\in\Nz$. Assume that the tensor bundles $V_j=V^{\sigma_j}_{\tau_j}:=\{T^{\sigma_j}_{\tau_j}\M,(\cdot|\cdot)_g\}$ with $j=1,2,3$ satisfy 
\begin{align}
\label{section 2: ptm-condition}
\sigma_3-\tau_3=\sigma_1+\sigma_2-\tau_1-\tau_2.
\end{align}
Suppose that $\mathsf{m}:V_1\times V_2\rightarrow V_3$ is a bundle multiplication, and $\vartheta_3=\vartheta_1+\vartheta_2$. Then $[(v_1,v_2)\mapsto \mathsf{m}(v_1,v_2)]$ is a bilinear and continuous map for $k\in\Nz$ and $s\leq k$
$$BC^{k,\vartheta_1}(\M, V_1)\times W^{s,\vartheta_2}_p(\M, V_2)\rightarrow W^{s,\vartheta_3}_p(\M, V_3).$$ 
\end{prop}
\begin{proof}
This follows from \cite[Theorem~13.5]{AmaAr}.
\end{proof}

\begin{prop}
\label{S2: nabla}
For $\F\in \{BC,W_p , \mathring{W}_p\}$, we have
$$\nabla\in \L(\F^{s+1,\vartheta}(\M, V^\sigma_\tau), \F^{s,\vartheta}(\M, V^\sigma_{\tau+1})).$$
\end{prop}
\begin{proof}
When $\F\in \{BC,W_p \}$, the case $s\in\Nz$ is immediate from the definition of the weighted function spaces. The non-integer case follows from \cite[Theorem~16.1]{AmaAr}. When $\F=\mathring{W}_p$, the assertion is an immediate consequence of its definition and a density argument. Indeed, for any $u\in \mathring{W}^{s+1}_p(\M,V^\sigma_\tau)$, there exists a sequence $(u_n)_n\in \mathcal{D}(\mathring{\M},V^\sigma_\tau)$ converging to $u$ in $W^{s+1}_p(\M, V^\sigma_\tau)$. Then the assertion for $\F=W_p$ implies that $(\nabla u_n)_n\in \mathcal{D}(\mathring{\M},V^\sigma_{\tau+1})$ converges to $\nabla u$ in $W^s_p(\M,V^\sigma_{\tau+1})$. Therefore, $\nabla u \in \mathring{W}^s_p(\M, V^\sigma_{\tau+1})$.
\end{proof}

\begin{prop}
\label{S2: change of wgt}
For $\F\in\{BC,W_p, \mathring{W}_p\}$, we have
$$f_{\vartheta}:=[u\mapsto \rho^{\vartheta}u] \in \Lis(\F^{s,\vartheta^\prime+\vartheta}(\M, V),\F^{s,\vartheta^\prime}(\M, V)).$$ 
\end{prop}
\begin{proof}
The case $\F=BC$ was shown in \cite[Proposition~2.6]{Shao14}. The proof for $\F=W_p$ follows in a similar manner. 
The remaining case, i.e., $\F=\mathring{W}_p$, is a direct consequence of its definition and a density argument as in the previous proposition.
\end{proof}


\subsection{\bf Surface divergence}
\begin{prop}
\label{S2: sharp-flat}
For $\F\in\{BC,W_p, \mathring{W}_p\}$, we have
$$[a \mapsto a^\sharp]\in \L(\F^{s,\vartheta}(\M, V^\sigma_{\tau+1}),\F^{s,\vartheta+2}(\M, V^{\sigma+1}_\tau)) .$$
$$[a \mapsto a_\flat]\in \L(\F^{s,\vartheta}(\M, V^{\sigma+1}_\tau),\F^{s,\vartheta-2}(\M, V^\sigma_{\tau+1})) .$$
\end{prop}
\begin{proof}
We only prove the second assertion. The first one follows in an analogous manner. 
For any $X\in T\M$, 
\begin{align*}
\nabla_X a_\flat=\nabla_X \langle g, a \rangle = \langle\nabla_X  g, a \rangle + \langle g, \nabla_X a \rangle= \langle g, \nabla_X a \rangle= (\nabla_X a )_\flat.
\end{align*}
The third equality follows from the metric preservation of the Levi-Civita connection.
This implies
$$
\nabla (a_\flat)= (\nabla a)_\flat.
$$
By induction, we have
$$ \nabla^k (a_\flat)= (\nabla^k a)_\flat.$$
Then the statement for the case $s\in\Nz$ is an immediate consequence of the definitions of the corresponding function spaces, and the non-integer case follows by interpolation theory and Definition~\eqref{S2: fractional Sobolev}.
\end{proof}

We denote by $\ev^{\sigma+1}_{\tau+1}: V^{\sigma+1}_{\tau+1}\rightarrow V^\sigma_\tau$ the contraction with respect to position $\sigma+1$ and $\tau+1$, that is for any $(i)\in\J^\sigma$, $(j)\in\J^\tau$ and $k,l\in\J^1$ and $\p\in\M$
$$
\ev^{\sigma+1}_{\tau+1}a  :=\ev^{\sigma+1}_{\tau+1} a^{(i;k)}_{(j;l)} \frac{\partial}{\partial x^{(i)}}\otimes \frac{\partial}{\partial x^k}\otimes dx^{(j)}\otimes dx^l:=a^{(i;k)}_{(j;k)} \frac{\partial}{\partial x^{(i)}} \otimes dx^{(j)} 
$$
in every local chart. Recall that the surface divergence of tensor fields with respect to the metric $g$ is the map
\begin{equation}
\label{S2: divergence}
\div=\div_g: C^1(\M,V^{\sigma+1}_\tau)\rightarrow C(\M, V^\sigma_\tau), \quad a\mapsto \ev^{\sigma+1}_{\tau+1}(\nabla a).
\end{equation}
\smallskip\\
Suppose that $\partial\M\neq \emptyset$. Since $T(\partial\M)$ is a subbundle of codimension $1$ of the vector bundle $(T\M)_{\partial\M}$ over $\partial\M$, there exists a unique vector field $\boldsymbol{n}$ in $(T\M)_{\partial\M}$ of length $1$ orthogonal to $T(\partial\M)$, and inward pointing. In every local coordinates, $\varphi_\kappa=\{x_1,\cdots,x_m\}$
$$\boldsymbol{n}=\frac{1}{\sqrt{g_{11}|\partial\Ok}} \frac{\partial}{\partial x^1}. $$

Put $V^\prime:=V^\tau_\sigma$. Let $\ev: V^{\sigma+\sigma_1}_{\tau+\tau_1} \times V^\prime \rightarrow V^{\sigma_1}_{\tau_1}$ denote the complete contraction. For any $a\in V^{\sigma+\sigma_1}_{\tau+\tau_1}$ and $b\in V^\prime$, the complete contraction (on the right)  is defined by
\begin{align*}
\ev(a,b)= a^{(i;i_1)}_{(j;j_1)} b^{(j)}_{(i)} \frac{\partial}{\partial x^{(i_1)}}\otimes dx^{(j_1)},
\end{align*}
with $(i)\in\mathbb{J}^\sigma$ ,$(i_1)\in\mathbb{J}^{\sigma_1}$, $(j) \in\mathbb{J}^\tau$,$(j_1)\in\mathbb{J}^{\tau_1}$, in local coordinates. 
The complete contraction (on the left) is defined in an analogous manner. Note that the complete contraction is a bundle multiplication.
\begin{theorem}
\label{S2: divergence thm}
For any $a\in W^{1,-\vartheta}_2(\M,V^\prime)$  and $b\in \mathring{W}^{1,\vartheta}_2(\M,V^{\sigma+1}_{\tau})$
$$-\int\limits_\M \langle \div b, a \rangle\, dV_g= \int\limits_\M \langle b, \nabla a \rangle\, dV_g.$$
\end{theorem}
\begin{proof}
By the divergence theorem and the density of $\mathcal{D}(\M,V^\prime)$ and $\mathcal{D}(\mathring{\M},V^{\sigma+1}_{\tau})$ in $ W^{1,-\vartheta}_2(\M,V^\prime)$ and $ \mathring{W}^{1,\vartheta}_2(\M,V^{\sigma+1}_{\tau})$, it suffices to show that 
\begin{align}
\label{S2: divergence eq}
\div(\ev(b,a))=\langle \div b, a\rangle+\langle b, \nabla a \rangle,
\end{align}
for any $a\in \mathcal{D}(\M,V^\prime)$  and $b\in \mathcal{D}(\mathring{\M},V^{\sigma+1}_{\tau})$.
Definition~\eqref{S2: divergence} yields
\begin{align*}
\div(\ev(b,a))&=\div(a^{(j)}_{(i)} b^{(i;k)}_{(j)} \frac{\partial}{\partial x^k})
= \partial_k (a^{(j)}_{(i)} b^{(i;k)}_{(j)}) +\Gamma^l_{lk}a^{(j)}_{(i)} b^{(i;k)}_{(j)},
\end{align*}
for $(i)\in\J^\sigma$, $(j)\in\J^\tau$. By \cite[formula~(3.17)]{Ama13} and \eqref{S2: divergence}
\begin{align*}
&\quad \langle \div b, a\rangle\\
&=a^{(j)}_{(i)} \partial_k(b^{(i;k)}_{(j)}) +(\sum\limits_{s=1}^\sigma \Gamma^{i_s}_{kh} b^{(i_1,\cdots,h,\cdots,i_\sigma;k)}_{(j)} -\sum\limits_{t=1}^\tau \Gamma_{k j_t}^{h} b^{(i;k)}_{(j_1,\cdots,h,\cdots,j_\tau)} +\Gamma_{k h}^k b^{(i;h)}_{(j)})a^{(j)}_{(i)},
\end{align*}
and 
\begin{align*}
\langle b, \nabla a \rangle=(\partial_k a^{(j)}_{(i)})b^{(i;k)}_{(j)} +(\sum\limits_{t=1}^\tau \Gamma_{k h }^{j_t} a^{(j_1,\cdots,h,\cdots,j_\tau)}_{(i)} -\sum\limits_{s=1}^\sigma \Gamma^{h}_{k i_s} a_{(i_1,\cdots,h,\cdots,i_\sigma)}^{(j)})b^{(i;k)}_{(j)}.
\end{align*}
This proves \eqref{S2: divergence eq}.
\end{proof}
\begin{cor}
\label{S2: divergence cor}
For any $a\in W^{1,-\vartheta+2\sigma-2\tau}_2(\M,V)$  and $b\in \mathring{W}^{1,\vartheta}_2(\M,V^{\sigma+1}_{\tau})$ 
$$-\int\limits_\M (  \div b | a )_g\, dV_g= \int\limits_\M (b | \gd a  )_g\, dV_g.$$
\end{cor}
\begin{proof}
In \cite[p. 10]{AmaAr}, it is shown that for any $X\in T\M$ and $a\in V$
$$\nabla_X (G^\tau_\sigma a) = G^\tau_\sigma(\nabla_X a). $$
Therefore,
$$\nabla G^\tau_\sigma a = (G^{\tau+1}_\sigma(\nabla u))_\flat. $$
Now it is an easy task to check
\begin{align}
\label{S2: nabla-grad}
\nabla G^\tau_\sigma u =(G^{\tau+1}_\sigma(\gd u)_\flat)_\flat= G^\tau_{\sigma+1}(\gd u).
\end{align}
This implies the asserted result.
\end{proof}

\begin{prop}
\label{S2: div-tangent}
$\div\in\L(\F^{s+1,\vartheta}(\M, V^{\sigma+1}_\tau), \F^{s,\vartheta}(\M, V^\sigma_\tau))$ for $\F\in \{BC,W_p, \mathring{W}_p\}$.
\end{prop}
\begin{proof}
Given any $a\in \F^{s,\vartheta}(\M, V^{\sigma+1}_\tau)$,
it is easy to see that
$$\| \Rc \ev^{\sigma+1}_{\tau+1} a\|_{l_q^\vartheta(\bF^s)} \leq C \|\Rc a\|_{l_q^\vartheta(\bF^s(E^{\sigma+1}_{\tau+1}))} $$
with $\kf(\pi_\kappa a^{(i;k)}_{(j;k)})$ in the $((i),(j))$ position.
Here $q=\infty$ for $\F=BC$, or $q=p$ for $\F\in \{W_p, \mathring{W}_p\}$.
Combining with Proposition~\ref{S2: retraction&coretraction}, it implies that
$$\ev^{\sigma+1}_{\tau+1}\in \L(\F^{s,\vartheta}(\M,V^{\sigma+1}_{\tau+1}), \F^{s,\vartheta}(\M, V)) .$$
Using Proposition~\ref{S2: nabla}, we can now prove the asserted result.
\end{proof}


\subsection{\bf Spaces of negative order}
For any $u\in \mathcal{D}(\M,V)$ and $v\in \mathcal{D}(\M,V^\prime)$, we put 
$$\langle u , v\rangle_\M:= \int\limits_\M \langle u,\bar{v} \rangle\, dV_g. $$
Then we define
\begin{equation}
\label{S2.3:def- neg order space}
W^{-s,\vartheta}_p(\M, V):=(\mathring{W}^{s,-\vartheta}_{p^\prime}(\M, V^\prime))^\prime
\end{equation}
by mean of the duality pairing $\langle \cdot ,\cdot \rangle_\M$. It is convenient to denote by $\mathring{W}^{-s,\vartheta}_p(\M, V)$ the closure of $\mathcal{D}(\mathring{\M},V)$ in $W^{-s,\vartheta}_p(\M, V)$. Then
\begin{align}
\label{S2: t<1/p}
\mathring{W}^{t,\vartheta}_p(\M, V)= W^{t,\vartheta}_p(\M, V),\quad t<1/p.
\end{align}
We refer  the reader to \cite[Section~12]{Ama13} for more details. Given $u\in \F^{-s,\vartheta}(\M, V)$ with $\F\in \{W_p, \mathring{W}_p\}$ and $v\in \mathcal{D}(\mathring{\M},V^{\tau+1}_\sigma)$
\begin{align}
\label{S2: def-nabla-neg}
\langle \nabla u , v\rangle_\M:= -\int\limits_\M \langle u, \div(\bar{v}) \rangle\, dV_g.
\end{align}
Theorem~\ref{S2: divergence thm} shows for $u\in \F^{-s,\vartheta}(\M,V^{\sigma+1}_\tau)$  and $v\in \mathcal{D}(\mathring{\M},V^\prime)$
$$\langle \div u , v\rangle_\M= -\int\limits_\M \langle u, \nabla \bar{v} \rangle\, dV_g. $$
By means of Proposition~\ref{S2: nabla} and \ref{S2: div-tangent}, it is not hard to prove the following proposition.
\begin{prop}
\label{S2: nabla-neg}
Suppose that $\F\in \{ W_p, \mathring{W}_p \}$. Then
$$\nabla \in \L(\F^{-s,\vartheta}(\M, V), \F^{-s-1,\vartheta}(\M,V^\sigma_{\tau+1})),$$ 
and 
$$\div \in \L(\F^{-s,\vartheta}(\M,V^{\sigma+1}_\tau), \F^{-s-1,\vartheta}(\M, V)).$$
\end{prop}

Let $\langle\cdot | \cdot\rangle_{2,\vartheta^\prime}$ be the inner product in $L_2^{\vartheta^\prime}(\M, V)$, that is,
\begin{align}
\label{S2.4: dual space}
\langle u | v\rangle_{2,\vartheta^\prime}:=\int\limits_\M \rho^{2\vartheta^\prime+2\tau-2\sigma}(  u | v )_g\, dV_g .
\end{align}
\begin{prop}
\label{S2: V-V'}
Suppose that $\F\in \{ W_p, \mathring{W}_p \}$ and $s\in\R$. Then
$$[u\mapsto \rho^{2\tau-2\sigma}G^\tau_\sigma u]\in \Lis(\F^{s,\vartheta}(\M, V), \F^{s,\vartheta}(\M,V^\prime)).$$
\end{prop}
\begin{proof}
The statement follows from Proposition~\ref{S2: change of wgt}, an analogue of the proof for Proposition~\eqref{S2: sharp-flat} and the open mapping theorem.
\end{proof}

In virtue of Proposition~\ref{S2: change of wgt} and \ref{S2: V-V'},  now one readily checks that 
\begin{align}
\label{S2.3: Lp-reflexive}
(\mathring{W}^{s,\vartheta}_p(\M, V))^\prime_{\vartheta^\prime}\doteq W^{-s,2\vartheta^\prime-\vartheta}_{p^\prime}(\M, V),
\end{align}
where $(\mathring{W}^{s,\vartheta}_p(\M, V))^\prime_{\vartheta^\prime}$ is the dual space of $W^{s,\vartheta}_p(\M, V)$ with respect to $\langle\cdot | \cdot\rangle_{2,\vartheta^\prime}$. 


\section{\bf $L_p$-theory of $(\rho,\lambda)$-singular elliptic operators}

Let $\sigma,\tau\in\Nz$ and $\lambda^\prime\in\R$. Suppose that $\cA: \mathcal{D}(\mathring{\M},V)\rightarrow \Gamma(\M,V)$ is a second order differential operator defined as follows.
\begin{align}
\label{S3: operator A}
\cA u:= -\div(\vec{a}\cdot\gd u)+\ev(\nabla u, a_1) +a_0 u,
\end{align}
with $\vec{a}\in C^1(\M,T^1_1\M)$, $a_1\in \Gamma(\M,T\M)$ and $a_0\in \C^\M$, 
for any $u\in C^\infty(\M, V)$ and some $\lambda\in\R$.
We put for all $\omega\geq 0$
\begin{align*}
\Aw u:= \cA u + \omega \rho^{-\lambda}u. 
\end{align*}
Center contraction $[u\mapsto\vec{a}\cdot\gd u]$ is defined by the relationship
$$\cdot: V^1_1 \times V^{\sigma+1}_{\tau}\rightarrow V^{\sigma+1}_{\tau}: (a,b)\mapsto a\cdot b , $$
and in every local chart for $\p\in \M$, we have
\begin{align*}
&(a\cdot b)(\p):= \{ a^{l}_{k} \frac{\partial}{\partial x^l} \otimes dx^k  (\p)\}  \cdot \{ b^{(i;h)}_{(j)} \frac{\partial}{\partial x^{(i)}}\otimes \frac{\partial}{\partial x^h}\otimes dx^{(j)} (\p) \}\\
&:=a^l_k b^{(i;k)}_{(j)} \frac{\partial}{\partial x^{(i)}}\otimes \frac{\partial}{\partial x^l}\otimes dx^{(j)} (\p)
\end{align*}
with $(i)\in\J^\sigma$, $(j)\in \J^\tau$ and $l,k,h\in\J^1$.
Here we write a differential operator in divergence form, which will benefit us in giving a precise bound for the constant $\omega$.

\subsection{\bf $L_2$-theory}
We impose the following assumptions on the coefficients of $\cA$ and the compensation term $\omega\rho^{-\lambda}$.
\begin{itemize}
\item[(A1)] $\cA$ is $(\rho,\lambda)$-regular, by which we means that $\vec{a}\in BC^{1,\lambda-2}(\M,T^1_1\M)$ is symmetric and 
$$a_1\in L_\infty^\lambda(\M, T\M),\quad a_0\in L_\infty^\lambda(\M).$$
\item[(A2)] $\cA$ is $(\rho,\lambda)$-singular elliptic. More precisely, there exists some $C_{\hat{\sigma}}>0$ such that
$$( \vec{a}\cdot\xi | \xi)_{g(\p)} \geq C_{\hat{\sigma}}\rho^{2-\lambda} |\xi|_{g}^2(\p), \quad \xi\in V^{\sigma+1}_\tau,\quad \p\in\M .$$
\item[(A3)] $\omega>\omega_\cA$, where $\omega_\cA\in \R$ satisfies for some $C_1<2$
\begin{align}
\label{S3: A3}
&\text{essinf}(\Rp(\rho^{\lambda}a_0) +\omega_\cA)>0; \\
\label{S3: A2}
&\rho^{\lambda-1} |(2\lambda^\prime+2\tau-2\sigma)\vec{a}\cdot\gd \log \rho+a_1|_g \leq C_1 \sqrt{C_{\hat{\sigma}}(\Rp(\rho^{\lambda}a_0) +\omega_\cA)};\\
\label{S3: A4}
&\rho^{\lambda-1} | (2\lambda^\prime-\lambda+2\tau-2\sigma)\vec{a}\cdot\gd \log \rho +a_1|_g \leq C_1 \sqrt{C_{\hat{\sigma}}(\Rp(\rho^{\lambda}a_0) +\omega_\cA)}. 
\end{align}
\end{itemize}
Here the weighted $L_\infty$-space $L^\lambda_\infty(\M,T\M)$ and $L_\infty^\lambda(\M)$ are defined in an obvious manner as for the weighted $L_p$-spaces in Section~2.2. 
We may replace the compensation term $\omega\rho^{-\lambda}$ by a largeness condition for the potential term $a_0$, which can be stated as follows.
\begin{itemize}
\item[(A3')] $\Rp(\rho^\lambda a_0)$ is so large that there exists some $C_1<2$ and $\omega_\cA<0$ such that
\begin{align*}
&\text{essinf}(\Rp(\rho^{\lambda}a_0)+\omega_\cA)>0; \\
&\rho^{\lambda-1} |(2\lambda^\prime+2\tau-2\sigma)\vec{a}\cdot\gd \log \rho+a_1|_g \leq C_1 \sqrt{C_{\hat{\sigma}}(\Rp(\rho^{\lambda}a_0) +\omega_\cA)};\\
&\rho^{\lambda-1} | (2\lambda^\prime-\lambda+2\tau-2\sigma)\vec{a}\cdot\gd \log \rho +a_1|_g \leq C_1 \sqrt{C_{\hat{\sigma}}(\Rp(\rho^{\lambda}a_0) +\omega_\cA)}. 
\end{align*}
\end{itemize}
Note that in (A3') only negative values $\omega_\cA$ are admissible, which is different from (A3).

Throughout, we assume that the singular data $[\![\rho]\!]$ and the constant $\lambda$ satisfy
\begin{align}
\label{S3: rho & lambda}
\left\{\begin{aligned}
&\|\rho\|_\infty \leq 1,\quad \lambda\geq 0, &&\text{or}&&\\
&\|\rho\|_\infty \geq 1,\quad \lambda\leq 0. &&
\end{aligned}\right.
\end{align}
Note that  the case $\lambda=0$ has been studied in \cite{Ama13b}.  In this case, actually no restriction for $\|\rho\|_\infty$ is required.

Since $W_p(\M,V)=\mathring{W}_p(\M,V)$ when $\partial\M=\emptyset$, in the sequel, we always focus on the space $\mathring{W}_p(\M,V)$.
Given $\lambda^\prime\in\R$, let $X:=\mathring{W}^{1,\lambda^\prime-\lambda/2}_2(\M, V)$. Then we can associate with $\Aw$ a form operator $\aw$ with $D(\aw)=X$, defined by 
\begin{align*}
\aw(u,v)
&= \langle \vec{a}\cdot\gd u | \gd v  \rangle_{2,\lambda^\prime}
+ \langle \ev(\nabla u, (2\lambda^\prime+2\tau-2\sigma)\vec{a}\cdot \gd \log \rho +a_1) | v \rangle_{2,\lambda^\prime}
\\
&\quad+\langle (a_0+\omega\rho^{-\lambda}) u | v \rangle_{2,\lambda^\prime} 
\end{align*}
for all $u,v\in X$. Recall that $\langle\cdot | \cdot\rangle_{2,\lambda^\prime}$ is the inner product in $L_2^{\lambda^\prime}(\M, V)$, see~\eqref{S2.4: dual space}.

\begin{lem}
\label{S3: pts-mul}
For any $\sigma,\tau, \sigma^\prime,\tau^\prime\in\Nz$, it holds that
\begin{itemize}
\item[(a)] $|\vec{a}\cdot\xi |_{g^{\tau}_{\sigma+1}}\leq |\vec{a}|_{g^1_1} |\xi|_{g^{\tau}_{\sigma+1}},\quad \xi\in V^{\sigma+1}_\tau.$
\item[(b)] $(a | b)_{g^\tau_\sigma} \leq |a|_{g^\tau_\sigma} |b|_{g^\tau_\sigma},\quad a,b\in V^\sigma_\tau. $
\item[(c)] $|\ev(a,b)|_{g_\sigma^\tau}\leq |a|_{g^{\sigma^\prime}_{\tau^\prime}} |b|_{g_{\sigma+\sigma^\prime}^{\tau+\tau^\prime}},\quad a\in V^{\tau^\prime}_{\sigma^\prime}, b\in V^{\sigma+\sigma^\prime}_{\tau+\tau^\prime}.$
\end{itemize}
\end{lem}
\begin{proof}
Statement (a) can be verified via direct computation.
Statements (b) and (c) follow  from identity \eqref{S2: G^tau_sigma} and \cite[formula~(A5)]{Ama13b}. 
\end{proof}

\begin{prop}
\label{S3: cont-L_2-coer}
$\aw$ is continuous and $X$-coercive. More precisely,
\begin{itemize}
\item[](Continuity) there exists some constant $M$ such that for all $u,v\in X$
$$|\aw(u,v)|\leq M\|u\|_X \|v\|_X ;$$
\item[]($X$-Coercivity) for $\omega$ large enough, there is some $M$ such that for any $u\in X$
$$\Rp(\aw(u,u)) \geq M \|u\|^2_X.$$
\end{itemize}
\end{prop}
\begin{proof}
(i) By \cite[formula~(5.8)]{Ama13b}, we have
\begin{align}
\label{S3: gd log r}
\gd \log \rho \in BC^{1,2}(\M,T\M).
\end{align}
Proposition~\ref{S2: pointwise multiplication}, (A1) and Lemma~\ref{S3: pts-mul} then imply that
$$\vec{a}\cdot \gd\log \rho= \ev(\vec{a}, \gd \log \rho)\in BC^{1,\lambda}(\M,T\M). $$
For any $u,v\in X$, 
\begin{align*}
&\quad|\aw(u,v)|\\
&\leq  \int\limits_\M \rho^{2\lambda^\prime+2\tau-2\sigma} |\vec{a}\cdot \gd u|_g |\gd v|_{g}\, dV_g\\
&\quad + \int\limits_\M \rho^{2\lambda^\prime+2\tau-2\sigma}|(2\lambda^\prime+2\tau-2\sigma)\vec{a}\cdot\gd \log \rho+a_1|_g |\nabla u|_{g} |v|_g \, dV_g\\
&\quad + \int\limits_\M \rho^{2\lambda^\prime+2\tau-2\sigma}(\rho^{\lambda}a_0+\omega)|\rho^{-\lambda/2}u|_g |\rho^{-\lambda/2}v|_g\, dV_g \\
& \leq \|\rho^{\lambda-2} \vec{a}\|_\infty^\M (\int\limits_\M |\rho^{\lambda^\prime+1-\lambda/2+\tau-\sigma}|\nabla u|_{g}|^2\, dV_g)^{1/2}(\int\limits_\M |\rho^{\lambda^\prime+1-\lambda/2+\tau-\sigma}|\nabla v|_{g}|^2\, dV_g)^{1/2}\\
&\quad + \|\rho^{\lambda-1} |(2\lambda^\prime+2\tau-2\sigma)\vec{a}\cdot\gd \log \rho+a_1|_g\|_\infty\\
&\quad
 (\int\limits_\M |\rho^{\lambda^\prime+1-\lambda/2+\tau-\sigma}|\nabla u|_{g}|^2\, dV_g)^{1/2}(\int\limits_\M |\rho^{\lambda^\prime-\lambda/2+\tau-\sigma}|v|_{g}|^2\, dV_g)^{1/2}\\
 &\quad + \|\rho^{\lambda}a_0 +\omega\|_\infty (\int\limits_\M |\rho^{\lambda^\prime-\lambda/2+\tau-\sigma}u|_g^2\, dV_g)^{1/2}(\int\limits_\M |\rho^{\lambda^\prime-\lambda/2+\tau-\sigma}v|_{g}^2\, dV_g)^{1/2}\\
 &\leq M(\omega)\|u\|_X\|v\|_X.
\end{align*}
This proves the continuity of $\aw$.

(ii) 
Given any $u\in X$, we have
\begin{align*}
&\quad \Rp(\aw(u,u))& \\
& \geq C_{\hat{\sigma}}\int\limits_\M |\rho^{\lambda^\prime+1-\lambda/2+\tau-\sigma}|\gd u|_{g}|^2\, dV_g  \\
&\quad- C_1\int\limits_\M \sqrt{C_{\hat{\sigma}}(\Rp(\rho^{\lambda}a_0)+\omega_\cA)}|\rho^{\lambda^\prime-\lambda/2+\tau-\sigma} u|_g|\rho^{\lambda^\prime+1-\lambda/2+\tau-\sigma} \nabla u|_{g}\, dV_g\\
&\quad + \int\limits_\M (\Rp(\rho^{\lambda} a_0) +\omega)|\rho^{\lambda^\prime-\lambda/2+\tau-\sigma}u|_g^2\, dV_g\\
&\geq (1-\frac{C_1^2}{4})C_{\hat{\sigma}}\int\limits_\M |\rho^{\lambda^\prime+1-\lambda/2+\tau-\sigma}|\nabla u|_{g}|^2\, dV_g +(\omega-\omega_\cA)\int\limits_\M |\rho^{\lambda^\prime-\lambda/2+\tau-\sigma}u|^2\, dV_g\\
&\geq M(\omega) \|u\|^2_X
\end{align*}
for all $\omega>\omega_\cA$ and some $M(\omega)>0$. In the second line, we have adopted Lemma~\ref{S3: pts-mul} and \eqref{S3: A2}.
\end{proof}
Proposition~\ref{S3: cont-L_2-coer} shows that $\aw$ with $D(\aw)=X$ is densely defined, sectorial and closed on $L_2^{\lambda^\prime}(\M, V)$. By \cite[Theorems~VI.2.1, IX.1.24]{Kato80}, there exists an associated operator $T$ such that $-T$ generates a contractive strongly continuous analytic semigroup on $L_2^{\lambda^\prime}(\M, V)$, i.e., $\|e^{-tT}\|_{\L(L_2^{\lambda^\prime}(\M, V))}\leq 1$ for all $t\geq 0$, with domain 
$$D(T):=\{u\in X, \exists ! v\in L_2^{\lambda^\prime}(\M, V):\aw(u,\phi)=\langle v | \phi \rangle_{2,\lambda^\prime}, \forall \phi\in X \},\quad T u=v, $$
which is a core of $\aw$.
$T$ is unique in the sense that there exists only one operator satisfying 
$$\aw(u,v)= \langle T u, v \rangle_{2,\lambda^\prime},\quad u\in D(T),\, v\in X.$$
On the other hand,  by \eqref{S2: nabla-grad} and definition~\eqref{S2: def-nabla-neg}, we can get
$$\langle \Aw u | v \rangle_{2,\lambda^\prime}= \aw(u,v),\quad u,v\in X.$$
So by the uniqueness of $T$, we have
$$\Aw|_{D(T)}=T .$$
Therefore, $-\Aw$ generates a contractive strongly continuous analytic semigroup on $L_2^{\lambda^\prime}(\M, V)$ with domain $D(\Aw)$:
$$D(\Aw):=\{u\in X, \exists! v\in L_2^{\lambda^\prime}(\M, V):\aw(u,\phi)=\langle v | \phi \rangle_{2,\lambda^\prime}, \forall \phi\in X \},\quad \Aw u=v. $$

In the rest of this subsection, our aim is to show that $D(\Aw)\doteq \mathring{W}^{2,\lambda^\prime-\lambda}_2(\M, V)$. 
Define 
$$\Bw u:=-\div(\rho^{\lambda} \vec{a}\cdot \gd  u) +  \ev(\nabla u, \rho^{\lambda} a_1) + (\rho^{\lambda} a_0 +\omega)u.$$

Recall an operator $A$ is said to belong to the class $\cH(E_1,E_0)$ for some densely embedded Banach couple $E_1\overset{d}{\hookrightarrow}E_0$, if $-A$ generates a strongly continuous analytic semigroup on $E_0$ with $dom(-A)=E_1$. 

(A1)-(A2) imply that
$$\rho^{\lambda} \vec{a}\in BC^{1,-2}(\M,T^1_1\M),\quad \rho^{\lambda}a_1\in L_\infty(\M,T\M),\quad \rho^{\lambda}a_0\in L_\infty(\M), $$
and
$$( \rho^{\lambda} \vec{a}\cdot \xi | \xi)_{g(\p)} \geq C_{\hat{\sigma}} \rho^2(\p)|\xi|_{g(\p)}^2,\quad \xi\in V^{\sigma+1}_\tau,\quad \p\in\M. $$
By \cite[Theorem~5.2]{Ama13b}, we obtain
\begin{align}
\label{S3.1: Max-reg-Bw}
\Bw\in \cH(\mathring{W}^{2,\lambda^\prime-\lambda}_2(\M, V), L^{\lambda^\prime-\lambda}_2(\M, V)).
\end{align}
Note that although  \cite[Theorem~5.2]{Ama13b} is only formulated for scalar functions, this theorem can be easily generalized to arbitrary tensor fields.

For any $u\in\mathcal{D}(\M,V)$, one checks that
\begin{align}
\label{S3: Aw & Bw}
\rho^{-\lambda}\Bw u= \Aw u -\lambda\ev(\nabla u, \vec{a}\cdot \gd \log \rho)=:\Aw u +\cP u.
\end{align}
It follows from Propositions~\ref{S2: pointwise multiplication}, \ref{S2: nabla}  and \eqref{S3: gd log r} that 
$$\cP \in \L(\mathring{W}^{1,\lambda^\prime-\lambda}_2(\M, V),  L_2^{\lambda^\prime}(\M, V)). $$
Combining with Proposition~\ref{S2: change of wgt}, we have
$$\rho^{\lambda}\cP \in \L(\mathring{W}^{1,\lambda^\prime-\lambda}_2(\M, V),  L_2^{\lambda^\prime-\lambda}(\M, V)). $$
Let $\fB_\omega:=\Bw-\rho^{\lambda}\cP$.
By well-known perturbation results of analytic semigroups and Definition~\eqref{S2: fractional Sobolev}, we infer that
\begin{align}
\label{S3: Bw-P}
\fB_\omega \in \cH(\mathring{W}^{2,\lambda^\prime-\lambda}_2(V), L^{\lambda^\prime-\lambda}_2(V)). 
\end{align}

Then for $\omega>\omega_\cA$, the previous discussion on $\Aw$ and \eqref{S3: A4} show that $-\fB_\omega$ generates a contractive strongly continuous analytic semigroup on $L_2^{\lambda^\prime-\lambda}(V)$. Then, together with \eqref{S3.1: Max-reg-Bw}, this implies that for $\omega$ sufficiently large,
$$\fB_\omega\in \Lis(D(\fB_\omega), L^{\lambda^\prime-\lambda}_2(V))\cap\Lis(\mathring{W}^{2,\lambda^\prime-\lambda}_2(V), L^{\lambda^\prime-\lambda}_2(V)).$$
Now we infer that $D(\fB_\omega)\doteq \mathring{W}^{2,\lambda^\prime-\lambda}_2(V)$.  
Observe that $D(\fB_\omega)$ is invariant for $\omega>\omega_\cA$. Thus for all $\omega>\omega_\cA$, the operator $-\fB_\omega$ generates a contractive strongly continuous analytic semigroup on $L_2^{\lambda^\prime-\lambda}(V)$ with domain $\mathring{W}^{2,\lambda^\prime-\lambda}_2(V)$.

\begin{theorem}
\label{S3: domain}
Suppose that the differential operator 
$$\cA u:= -\div(\vec{a}\cdot\gd u)+\ev(\nabla u, a_1) +a_0 u,$$ 
is $(\rho,\lambda)$-regular and $(\rho,\lambda)$-singular elliptic, and  
the constant $\omega$ satisfies  (A3). 
Define $\Aw:=\cA +\omega \rho^{-\lambda}.$
Then
$$\Aw\in \cH(\mathring{W}^{2,\lambda^\prime-\lambda}_2(\M, V), L_2^{\lambda^\prime}(\M, V)) \cap \Lis(\mathring{W}^{2,\lambda^\prime-\lambda}_2(\M, V), L_2^{\lambda^\prime}(\M, V)), $$
and the semigroup $\{e^{-t\Aw}\}_{t\geq 0}$ is contractive.
\end{theorem}
\begin{proof}
By Propositions~\ref{S2: pointwise multiplication}, \ref{S2: nabla}, \ref{S2: div-tangent} and Lemma~\ref{S3: pts-mul}, we obtain
$$\Aw \in\L(\mathring{W}^{2,\lambda^\prime-\lambda}_2(\M, V),L_2^{\lambda^\prime}(\M,V)). $$
This implies together with the definition of $D(\Aw)$ that 
$$\mathring{W}^{2,\lambda^\prime-\lambda}_2(\M, V) \hookrightarrow D(\Aw). $$
We have shown that for $\omega>\omega_\cA$, 
\begin{align}
\label{S3: Bw-Lis}
\Aw=\rho^{-\lambda}\fB_\omega \in\Lis(\mathring{W}^{2,\lambda^\prime-\lambda}_2(\M, V),L_2^{\lambda^\prime}(\M, V)).
\end{align}
Now by \eqref{S3: Bw-Lis}, we can establish
$$D(\Aw)\doteq \mathring{W}^{2,\lambda^\prime-\lambda}_2(\M, V). $$
The asserted statement thus follows.
\end{proof}
\begin{cor}
Suppose that $\cA$ is $(\rho,\lambda)$-regular and $(\rho,\lambda)$-singular elliptic, and  satisfies  (A3'). Then 
$$\cA\in \cH(\mathring{W}^{2,\lambda^\prime-\lambda}_2(\M, V), L_2^{\lambda^\prime}(\M, V)) \cap \Lis(\mathring{W}^{2,\lambda^\prime-\lambda}_2(\M, V), L_2^{\lambda^\prime}(\M, V)), $$
and the semigroup $\{e^{-t\cA}\}_{t\geq 0}$ is contractive.
\end{cor}

\subsection{\bf $L_p$-theory for scalar functions}
In this subsection, we assume that $V=\C$ and abbreviate the corresponding functions space to be $\mathring{W}^{s,\vartheta}_p(\M)$. The aim of this subsection is to prove that the differential operator $\Aw$ generates a contractive strongly continuous analytic semigroup on $L_p^{\lambda^\prime}(\M)$ with $1<p<\infty$ for large $\omega$.

We first show the following Riesz-Thorin interpolation theorem for the weighted $L_p$-spaces with $1\leq p\leq \infty$.
\begin{lem}
\label{S3.2: Riesz-Thorin}
Let $1\leq p_0 <p_1\leq \infty$, $\theta\in (0,1)$, and $\vartheta\in\R$. Define $\frac{1}{p_\theta}=\frac{1-\theta}{p_0} +\frac{\theta}{p_1}$. Then for every $f\in L_{p_0}^\vartheta(\M)\cap L_{p_1}^\vartheta(\M)$,
$$\|f\|_{L_{p_\theta}^\vartheta} \leq \|f\|_{L_{p_0}^\vartheta}^{1-\theta} \|f\|_{L_{p_1}^\vartheta}^\theta. $$ 
\end{lem}
\begin{proof}
Observe that the operator $f_\vartheta$ defined in Proposition~\ref{S2: change of wgt} is indeed an isometry from $L_p^\vartheta(\M)$ to $L_p(\M)$ for $1\leq p\leq \infty$. Then we have
\begin{align*}
\|f\|_{L_{p_\theta}^\vartheta}= \|\rho^\vartheta f\|_{L_{\p_\theta}}\leq  \|\rho^\vartheta f\|_{L_{p_0}}^{1-\theta} \|\rho^\vartheta f\|_{L_{p_1}}^\theta = \|f\|_{L_{p_0}^\vartheta}^{1-\theta} \|f\|_{L_{p_1}^\vartheta}^\theta.
\end{align*}
\end{proof}

The adjoint, $\Awla$, of $\Aw$ with respect to $L_2^{\lambda^\prime/2}(\M)$ can be easily computed as follows.
$$\Awla u = -\div(\vec{a}\cdot \gd u) -\ev(\nabla u , 2\lambda^\prime\vec{a}\cdot\gd \log \rho+\bar{a}_1 ) + (b(\lambda^\prime,\ba) +\omega \rho^{-\lambda})u, $$
where with $\ba:=(\vec{a},a_1,a_0)$
\begin{align*}
b(\lambda^\prime,\ba):= \bar{a}_0 - \div(\lambda^\prime \vec{a}\cdot\gd\log\rho+\bar{a}_1)-\lambda^\prime(\lambda^\prime \vec{a}\cdot\gd\log\rho+\bar{a}_1 |  \gd\log\rho)_g.
\end{align*}
Here we have used the equality 
$$
\langle \ev( \nabla u, a) | v \rangle_{2,\lambda^\prime/2}= -\langle u | \ev(\nabla v, \bar{a})\rangle_{2,\lambda^\prime/2}- \langle u |(\div\bar{a} +\lambda^\prime (\bar{a}| \gd\log\rho)_g ) v\rangle_{2,\lambda^\prime/2} 
$$
for $a\in C^1(\M,T\M)$ and $u,v \in \mathcal{D}(\mathring{\M})$.

The adjoint, $\Awl:=(\Awla)^*$, of $\Awla$ with respect to $L_2 (\M)$ is 
\begin{align*}
\Awl u =-\div(\vec{a}\cdot \gd u) +\ev(\nabla u , 2\lambda^\prime\vec{a}\cdot\gd \log \rho+a_1 ) + (\tilde{b}(\lambda^\prime,\ba) +\omega \rho^{-\lambda})u,
\end{align*}
where
\begin{align*}
\tilde{b}(\lambda^\prime,\ba):= a_0 + \div(\lambda^\prime \vec{a}\cdot\gd\log\rho)-\lambda^\prime(\lambda^\prime \vec{a}\cdot\gd\log\rho+a_1 |  \gd\log\rho)_g.
\end{align*}

We impose the following conditions on the compensation term $\omega \rho^{-\lambda}$.
\begin{itemize}
\item[(A4)] $\omega>\omega_\cA$, where $\omega_\cA\in\R$ satisfies for some $C_1<2$
\begin{align*}
&\text{essinf}(\Rp(\rho^{\lambda}b(\lambda^\prime,\ba) +\omega_\cA)>0;\\
&\rho^{\lambda-1} |2\lambda^\prime \vec{a}\cdot\gd \log \rho+a_1|_g \leq C_1 \sqrt{C_{\hat{\sigma}}(\Rp(\rho^{\lambda}b(\lambda^\prime,\ba)) +\omega_\cA)};\\
&\text{essinf}(\Rp(\rho^{\lambda}\tilde{b}(\lambda^\prime,\ba) +\omega_\cA)>0;\\
&\rho^{\lambda-1} |2\lambda^\prime \vec{a}\cdot\gd \log \rho+a_1|_g \leq C_1 \sqrt{C_{\hat{\sigma}}(\Rp(\rho^{\lambda}\tilde{b}(\lambda^\prime,\ba)) +\omega_\cA)},
\end{align*}
\end{itemize}
and
\begin{itemize}
\item[(A5)] $\omega>\omega_\cA$, where $\omega_\cA\in\R$ satisfies for some $C_1<2$
\begin{align*}
&\text{essinf}(\Rp(\rho^{\lambda}\tilde{b}(\lambda^\prime-\lambda,\ba) +\omega_\cA)>0;\\
&\rho^{\lambda-1} |(2\lambda^\prime -\lambda) \vec{a}\cdot\gd \log \rho+a_1|_g \leq C_1 \sqrt{C_{\hat{\sigma}}(\Rp(\rho^{\lambda}b(\lambda^\prime,\ba)) +\omega_\cA)};\\
&\rho^{\lambda-1} |(2\lambda^\prime -\lambda) \vec{a}\cdot\gd \log \rho+a_1|_g \leq C_1 \sqrt{C_{\hat{\sigma}}(\Rp(\rho^{\lambda}\tilde{b}(\lambda^\prime-\lambda,\ba)) +\omega_\cA)}.
\end{align*}
\end{itemize}
We can also formulate an analogue of (A3') for the largeness of the potential term $a_0$ to replace the compensation condition~(A4) and (A5).
\smallskip\\
Then the discussion in Section~3.1 and (A4) imply that $-\Awla$ and $-\Awl$ generate contractive strongly continuous analytic semigroups on $L_2(\M)$ for all $\omega$ satisfying (A4). 

\begin{definition}
Let $q\in [1,\infty]$ and $\vartheta\in\R$.
A strongly continuous semigroup $\{T(t)\}_{t\geq 0}$ on $L_2^\vartheta(\M)$ is said to be $L_q^\vartheta$-contractive if
$$ \|T(t) u\|_{0,q;\vartheta} \leq \|u\|_{0,q;\vartheta},\quad t\geq 0,\quad   u\in L_2^\vartheta(\M)\cap L_q^\vartheta(\M).$$
\end{definition}

\begin{theorem}
\label{S3.2: Lp-continuity}
Suppose that the differential operator 
$$\cA u:= -\div(\vec{a}\cdot\gd u)+\ev(\nabla u, a_1) +a_0 u,$$ 
is $(\rho,\lambda)$-regular and $(\rho,\lambda)$-singular elliptic. 
For $\omega$ satisfying (A3)-(A5),
define $\Aw:=\cA +\omega \rho^{-\lambda}.$
Then
$$\Aw\in\cH(\mathring{W}^{2,\lambda^\prime-\lambda}_p(\M),L_p^{\lambda^\prime}(\M)) \cap \Lis(\mathring{W}^{2,\lambda^\prime-\lambda}_p(\M),L_p^{\lambda^\prime}(\M)), \quad 1<p<\infty ,$$
and the semigroup $\{e^{-t\Aw}\}_{t\geq 0}$ is contractive.
\end{theorem}
\begin{proof}
(i)
By  Proposition~\ref{S2: retraction&coretraction}, it is not hard to verify that $u\in X$ implies $(|u|-1)^+  \sg u \in X$ and
\begin{align}
\label{S3: der-(|u|-1)+}
\nabla ((|u|-1)^+  \sg u)=
\begin{cases}
\nabla u, \quad & |u|>1;\\
0, & |u|\leq 1.
\end{cases}
\end{align}
Here it is understood that 
\begin{align*}
\sg u:=
\begin{cases}
u/|u|, \quad & u\neq 0;\\
0, &u=0.
\end{cases}
\end{align*}
Now following a similar  proof to  step (ii) of Proposition~\ref{S3: cont-L_2-coer}, we get
$$\Rp(\aw(u,(|u|-1)^+  \sg u))\geq 0 , \quad \omega >\omega_\cA.$$
By \cite[Theorem~2.7]{Ouh92}, the semigroup $\{e^{-t\Aw}\}_{t\geq 0}$ is $L_\infty$-contractive.

Similarly, based on (A4), we can show that $\{e^{-t\Awla}\}_{t\geq 0}$ and $\{e^{-t\Awl}\}_{t\geq 0}$ are $L_\infty$-contractive as well. By a well-known argument, see \cite[Chapter~1.4]{Dav89}, this implies that for each $1<p<\infty$, $\{e^{-t\Awla}\}_{t\geq 0}$ and $\{e^{-t\Awl}\}_{t\geq 0}$ can be extended to contractive strongly continuous analytic semigroups on $L_p(\M)$ with angle 
$$ \theta_p\geq \theta(1-|2/p-1|), $$
where $\theta$ is the smaller one of the angles of the semigroups on $L_2(\M)$ generated by $\{e^{-t\Awla}\}_{t\geq 0}$ and $\{e^{-t\Awl}\}_{t\geq 0}$.

(ii)
Pick $v\in L_2^{\lambda^\prime}(\M)\cap L_1^{\lambda^\prime}(\M)$ and $u\in L_2(\M)\cap L_\infty(\M)$. 
We then have
\begin{align}
\label{S3.2: duality}
\notag |\langle e^{-t\Aw}v | u \rangle_{2,{\lambda^\prime/2}}| &=|\langle v | e^{-t\Awla}u \rangle_{2,{\lambda^\prime/2}}|
=|\langle \rho^{\lambda^\prime} v | e^{-t\Awla} u \rangle_{2,0}|\\
\notag&\leq \|v\|_{L_1^{\lambda^\prime}} \|e^{-t\Awla}u \|_{L_\infty}\\
&\leq \|v\|_{L_1^{\lambda^\prime}} \|   u \|_{L_\infty} .
\end{align}
We have thus established the $L_1^{\lambda^\prime}$-contractivity of the semigroup $\{e^{-t\Aw}\}_{t\geq 0}$.
It is then an immediate consequence of Lemma~\ref{S3.2: Riesz-Thorin} that $\{e^{-t\Aw}\}_{t\geq 0}$  is $L_p^{\lambda^\prime}$-contractive for  $1\leq p\leq 2$. 

(iii)
Now we modify a widely used argument, see \cite[Chapter~1.4]{Dav89}, for weighted $L_p$-spaces.
Choose $u\in L_2^{\lambda^\prime}(\M)$ with $\supp(u)\subset K$ with $K\subset \M$ satisfying $V_g(K)<\infty$. Then
\begin{align}
\label{S3.2: Chi_K}
\notag\lim\limits_{t\to 0^+} \| \chi_K e^{-t\Aw}u\|_{L_1^{\lambda^\prime}}
&=\lim\limits_{t\to 0^+} \langle \rho^{-\lambda^\prime}\chi_K | |e^{-t\Aw} u| \rangle_{2,{\lambda^\prime}}\\
&= \langle \rho^{-\lambda^\prime}\chi_K | |u| \rangle_{2,{\lambda^\prime}} = \|u\|_{L_1^{\lambda^\prime}}
\end{align}
by the strong $L_2^{\lambda^\prime}$-continuity of $\{e^{-t\Aw}\}_{t\geq 0}$. 
On the other hand, we also have 
$$\|e^{-t\Aw} u\|_{L_1^{\lambda^\prime}}\leq \|u\|_{L_1^{\lambda^\prime}}.$$
This together with \eqref{S3.2: Chi_K} implies that
$$\lim\limits_{t\to 0^+} \|\chi_{\M\setminus K} e^{-t\Aw} u\|_{L_1^{\lambda^\prime}}=0. $$
Now one can compute that 
\begin{align*}
\lim\limits_{t\to 0^+} \| e^{-t\Aw} u-u\|_{L_1^{\lambda^\prime}} &\leq \lim\limits_{t\to 0^+} \|\chi_K (e^{-t\Aw}u-u)\|_{L_1^{\lambda^\prime}}\\
&\leq \lim\limits_{t\to 0^+} \| e^{-t\Aw}u-u\|_{L_2^{\lambda^\prime}} \mu(K)^{1/2}=0.
\end{align*}
The set of such $u$ contains $\mathcal{D}(\M)$ and thus is dense in $L_1^{\lambda^\prime}(\M)$. This establishes the strong continuity of $\{e^{-t\Aw}\}_{t\geq 0}$ on $L_2^{\lambda^\prime}(\M)\cap L_1^{\lambda^\prime}(\M)$. 
Lemma~\ref{S3.2: Riesz-Thorin} then implies the strong continuity of $\{e^{-t\Aw}\}_{t\geq 0}$ on $L_2^{\lambda^\prime}(\M)\cap L_p^{\lambda^\prime}(\M) $ for $1\leq p\leq 2$.

By \eqref{S2.3: Lp-reflexive}, $L_p^{\lambda^\prime}(\M)$ is reflexive for $1<p<\infty$.
The strong continuity of $\{e^{-t\Aw}\}_{t\geq 0}$ on $L_2^{\lambda^\prime}(\M)\cap L_p^{\lambda^\prime}(\M) $ for $2< p<\infty$ now follows from
\cite[Theorem~1.4.9]{Gold85} and the strong continuity of $\{e^{-t\Awla}\}_{t\geq 0}$ on $L_q(\M)$ with $1<q<2$.

(iv)
Assume that $\{e^{-t\Aw}\}_{t\geq 0}$ is analytic on $L_2^{\lambda^\prime}(\M)$ with angle $\phi$. We define  $$H_z:=\rho^{\lambda^\prime} e^{-\Aw h(z)} \rho^{-\lambda^\prime},\quad \text{on} \quad S:=\{z\in\C: 0\leq \Rp z\leq 1\},$$ 
where $h(z):=re^{i\theta z}$ with $r>0$ and $|\theta|<\phi$. Then given any $u\in L_2(\M) \cap L_1(\M)$ and $v\in L_2(\M) \cap L_\infty(\M)$, we have
\begin{align}
\label{S3.2: Stein}
\notag |\langle H_z u | v \rangle_{2,0}|&\leq \|H_z u\|_{L_2} \|v\|_{L_2}\leq \|e^{\Aw h(z)}\rho^{-\lambda^\prime}u\|_{L_2^{\lambda^\prime}} \|v\|_{L_2}\\
& \leq \|\rho^{-\lambda^\prime}u\|_{L_2^{\lambda^\prime}} \|v\|_{L_2} \leq \|u\|_{L_2} \|v\|_{L_2}
\end{align}
for $z\in S$. Similarly, one can verify that $\langle H_z u | v \rangle_{2,0}$ is continuous on $S$ and analytic inside $\mathring{S}$. Moreover, 
\begin{align*}
|\langle H_z u | v \rangle_{2,0}|\leq
\|u\|_{L_1} \|v\|_{L_\infty},\quad \text{if }\Rp z=0. 
\end{align*}
By the Stein interpolation theorem, see \cite[Section~1.1.6]{Dav89}, we conclude that for all $0<t<1$, $u\in L_2^{\lambda^\prime}(\M) \cap L_1^{\lambda^\prime}(\M)$ and $\frac{1}{p}=1-t +\frac{t}{2}$,
$$ \| \rho^{\lambda^\prime}u\|_{L_p}\geq \| H_t \rho^{\lambda^\prime} u \|_{L_p}= \| \rho^{\lambda^\prime} e^{-\Aw h(t)}  u\|_{L_p}=  \|e^{-\Aw h(t)}  u\|_{L_p^{\lambda^\prime}}  .$$
Therefore, $\{e^{-t\Aw}\}_{t\geq 0}$ can be extended to a contractive strongly continuous analytic semigroup on $L_p^{\lambda^\prime}(\M)$ with angle $\phi(2-2/p)$ for $1<p<2$. 

When $2<p<\infty$, the analytic extension of $\{e^{-t\Aw}\}_{t\geq 0}$ follows from a duality argument as in \eqref{S3.2: duality}.

(v)
In order to determine the domain of $\Aw$,  we apply a similar discussion to the proof for Theorem~\ref{S3: domain}. 
We consider the adjoint, $\fB_\omega^*(\lambda^\prime-\lambda)$, of $\fB_\omega=\Bw -\rho^\lambda \mathcal{P}_\lambda$ with respect to $L_2^{(\lambda^\prime-\lambda)/2}(\M)$, i.e.,
$$\fB_\omega^*(\lambda^\prime-\lambda)u= -\div(\rho^\lambda \vec{a}\cdot\gd u) -\rho^\lambda\ev(\nabla u, (2\lambda^\prime-\lambda) \vec{a}\cdot\gd\log\rho +\bar{a}_1) +\rho^\lambda b(\lambda^\prime, \boldsymbol{\vec{a}})u, $$
and the adjoint, $\fB_\omega(\lambda^\prime-\lambda)$, of $\fB_\omega^*(\lambda^\prime-\lambda)$ with respect to $L_2(\M)$, i.e.,
\begin{align*}
\fB_\omega(\lambda^\prime-\lambda)u=& -\div(\rho^\lambda \vec{a}\cdot\gd u) +\rho^\lambda\ev(\nabla u, (2\lambda^\prime-\lambda) \vec{a}\cdot\gd\log\rho +a_1)\\
 &+\rho^\lambda \tilde{b}(\lambda^\prime-\lambda, \boldsymbol{\vec{a}})u.
\end{align*}

Following Step (i)-(iv), under Assumptions (A3) and (A5), 
we can show that $\fB_\omega$ generates a contractive strongly continuous analytic semigroup on $L_p^{\lambda^\prime-\lambda}(\M)$ for any $1<p<\infty$. By \cite[Theorem~5.2]{Ama13b}, for $\omega$ large enough, 
$$ \fB_\omega \in \cH(\mathring{W}^{2,\lambda^\prime-\lambda}_p(\M), L^{\lambda^\prime-\lambda}_p(\M)) \cap \Lis(\mathring{W}^{2,\lambda^\prime-\lambda}_p(\M), L^{\lambda^\prime-\lambda}_p(\M)).$$
An analogous argument to the proof for Theorem~\ref{S3: domain} and the discussion prior to this proof
yields that 
$$ \Aw= \rho^{-\lambda}\fB_\omega \in \cH(\mathring{W}^{2,\lambda^\prime-\lambda}_p(\M), L^{\lambda^\prime}_p(\M)) \cap \Lis(\mathring{W}^{2,\lambda^\prime-\lambda}_p(\M), L^{\lambda^\prime}_p(\M)).$$
\end{proof}

\begin{remark}
The proof of $L_\infty$-contractivity for unweighted $L_p$-spaces in \cite[Theorem~2.7]{Ouh92} suggests that there seems to be a more straightforward way to prove $L_\infty^{\lambda^\prime}$-contractivity. 
\smallskip\\
In fact, we can show that
$\{e^{-t\Aw} \}_{t\geq 0}$ is $L_\infty^{\lambda^\prime}$-contractive if for any $u\in X$
\begin{itemize}
\item[(i)] $(|u|-\rho^{-\lambda^\prime})^+ \sg u \in X$, and
\item[(ii)] $\Rp \aw(u , (|u|-\rho^{-\lambda^\prime})^+ \sg u)\geq 0$.
\end{itemize}
However, Condition~(ii), in general, does not hold for all $u\in X$.
\end{remark}

\begin{remark}
When the tensor field $V\neq \C$, it requires much more effort to establish the $L_p$-semigroup theory for the differential operator 
$$\cA u:= -\div(\vec{a}\cdot\gd u)+\ev(\nabla u, a_1) +a_0 u.$$
The author is not aware of how to obtain the $L_\infty^{\lambda^\prime}$-contractivity of the semigroup $\{e^{-t\Aw}\}_{t\geq 0}$. 
Instead, one needs to go through the local expressions of $\Aw$ and establish a similar contractivity property for these local expressions, and then prove generation of analytic semigroups of the local expressions.
However, the drawback of this technique  is reflected by the fact that it is hard to determine the precise bound for the constant $\omega$.
Indeed, we only know that for $\omega$ sufficiently large and $1<p<\infty$
$$\Aw\in\cH(\mathring{W}^{2,\lambda^\prime-\lambda}_p(\M, V),L_p^{\lambda^\prime}(\M, V)) \cap \Lis(\mathring{W}^{2,\lambda^\prime-\lambda}_p(\M, V),L_p^{\lambda^\prime}(\M, V)),$$
and the semigroup $\{e^{-\Aw}\}_{t\geq 0}$ is bounded on $L_p^{\lambda^\prime}(\M, V)$.
Because it is hard to apply this result, a rigorous proof for this assertion will not be stated in this article.
\end{remark}


\section{\bf Singular manifolds of pipe and wedge type and manifolds with holes}

As was shown by the examples in \cite{Ama14}, we can find manifolds with singularities of arbitrarily high dimension. Among them, a very important family is the {\em singular manifolds} of pipe and wedge type.

Following \cite{Ama14}, throughout we write $J_0:=(0,1]$ and $J_\infty:=[1,\infty)$, and assume $J\in \{J_0,J_\infty \}$. 
\smallskip\\
We denote by $\mathscr{R}(J)$ the set of all $R\in C^\infty(J,(0,\infty))$ with $R(1)=1$ such that 
$R(\alpha):=\lim\limits_{t\to \alpha} R(t)$ exists in $[0,\infty]$ if $J=J_\alpha$ with $\alpha\in\{0,\infty\}$. We write $R\in \mathscr{C}(J)$ if
\begin{align}
\label{S4: cusp ch}
\begin{cases}
\text{(i)} \quad & R\in \mathscr{R}(J), \text{ and }R(\infty)=0 \text{ if }J=J_\infty;\\
\text{(ii)}  & \int\limits_J dt/R(t)=\infty;\\
\text{(iii)}  & 	\|\partial^k_t R\|_{\infty} <\infty,\quad k\geq 1.
\end{cases}
\end{align}
The elements in $\mathscr{C}(J)$ are called {\em cusp characteristics} on $J$.

The following results from \cite{Ama14} are the cornerstones of the construction of {\em singular manifolds} of pipe and wedge type.
\begin{lem}
\label{S4: lem 4.1}
{\cite[Theorem~3.1]{Ama14}}
Suppose that $\rho$ is a bounded singularity function on $(\M,g)$, and $\tilde{\rho}$ is one for $(\tilde{\M},\tilde{g})$. Then $\rho\otimes\tilde{\rho}$ is a singularity function for $(\M\times\tilde{\M}, g+\tilde{g})$.
\end{lem}

\begin{lem}
\label{S4: lem 4.2}
{\cite[Lemma~3.4]{Ama14}}
Let $f: \tilde{\M}\rightarrow\M$ be a diffeomorphism of manifolds. Suppose that $(\M,g;\rho)$ is a singular manifold. Then so is
$(\tilde{\M},f^*g;f^*\rho)$.
\end{lem}

\begin{lem}
\label{S4: lem 4.3}
{\cite[Lemma~5.2]{Ama14}}
Suppose that $R\in \mathscr{C}(J)$. Then $R$ is a singularity function for $(J, dt^2)$.
\end{lem}

Assume that $(B,g_B;b)$ is a $d$-dimensional singular submanifold of $\R^{\bar{d}}$ with singularity function $b$, and $R\in \mathscr{C}(J)$. The (model) $(R,B)$-pipe $P(R,B)$ on $J$, also called $R$-pipe over $B$ on $J$, is defined by 
$$P(R,B)=P(R,B;J):=\{(t,R(t)y):\, t\in J, \, y\in B\}\subset \R^{1+\bar{d}} .$$
It is a $(1+d)$-dimensional submanifold of $\R^{1+\bar{d}}$. An $R$-pipe is an $R$-cusp if $R(\alpha)=0$ with $\alpha\in\{0,\infty\}$. The map
$$\phi_P=\phi_P(R):P \rightarrow J\times B: \quad (t,R(t)y)\rightarrow (t,y) $$
is a diffeomorphism, the {\em canonical stretching diffeomorphism} of $P$.

Then the above three lemmas show
\begin{lem}
$(P(R,B), \phi_P^*(dt^2 + g_B); \phi_P^*(R\otimes b))$ is a singular manifold.
\end{lem}

Assume that $(\Gamma, g_\Gamma)$ is a compact connected Riemannian manifold without boundary. Then the (model) $\Gamma$-wedge over the $(R,B)$-pipe, $P(R,B)$, is defined by
$$W=W(R,B,\Gamma):=P(R,B)\times\Gamma.$$
If $\Gamma$ is a one-point space, then $W$ is naturally identified with $P$. Thus every pipe is also a wedge.

Lemmas~\ref{S4: lem 4.1}-\ref{S4: lem 4.3} yield
\begin{lem}
\label{S4: wedge}
$(W(R,B,\Gamma), \phi_P^*(dt^2 + g_B)+g_\Gamma; \phi_P^*(R\otimes b)\otimes {\bf 1}_\Gamma)$ is a singular manifold.
\end{lem}

Another interesting class of manifolds is those with holes.
\begin{lem}
\label{S4: holes}
Suppose that $(\sM,g)$ is a  uniformly regular Riemannian manifold, and 
${\bf \Sigma}=\{\Sigma_1,\cdots,\Sigma_k\}$ is a finite set of disjoint $m$-dimensional compact manifolds with boundary such that $\Sigma_j\subset \mathring{\sM}$.
Put 
$$\M:=\sM\setminus \cup_{j=1}^k \Sigma_j.$$
Then
$(\M,g)$
is a singular manifold.
\end{lem}
\begin{proof}
This lemma immediately follows from \cite[Theorem~1.6]{Ama14}.
\end{proof}
We will show in Proposition~\ref{S5: sing mdf with holes} below how to choose a singularity function for such $(\M,g)$.


\section{\bf Differential operators on singular manifolds with $\sH_\lambda$-ends}

\subsection{\bf Differential operators on singular manifolds with {\em property $\sH_\lambda$}}
In the first subsection, we will exhibit a technique to remove the ``largeness" assumption on the potential term or the compensation term $\omega\rho^{-\lambda}$.

Suppose that $(\M,g;\rho)$ is a {\em singular manifold}. 
Without loss of generality, we assume that $\M$ is connected. 
Before beginning the discussion of any particular model, we first consider a variant of the operator $\cA$ defined in \eqref{S3: operator A}, i.e., 
\begin{align*}
\cA u:= -\div(\vec{a}\cdot\gd u)+\ev(\nabla u, a_1) +a_0 u.
\end{align*}
Put $v:=e^{-zh}u$ for some $z=a+ib\in\C$ with $|z|=1$, and $h\in C^2(\M,\R)$.
Then
\begin{align}
\label{S5: a tech}
\notag\cA u =& -\div(\vec{a} \cdot \gd (e^{zh} v)) + \ev(\nabla e^{zh} v, a_1 ) +e^{zh}a_0 v\\
\notag=&-\div(e^{zh} \vec{a} \cdot\gd v) -z\div(\vec{a}\cdot (e^{zh} v\otimes \gd h))+e^{zh} \ev(\nabla v, a_1) \\
\notag&+z e^{zh} \ev(\nabla h, a_1 )v
+e^{zh}a_0 v\\
\notag=&  e^{zh} \{ \cA v - 2z \ev(\nabla v, \vec{a} \cdot\gd h) \\
& -[z \div(\vec{a}\cdot\gd h)+ z^2(\vec{a} \cdot\gd h |\gd h)_g -z \ev(\nabla h, a_1)]v \}.
\end{align}
 
In the sequel, we let $\vec{a}:=\rho^{2-\lambda} g_\flat$, which means that
we will consider differential operators of the following form
$$\cA u:= -\div(\rho^{2-\lambda} \gd u) +\ev(\nabla u,a_1) + a_0 u $$
with $\rho$ and $\lambda$ satisfying \eqref{S3: rho & lambda}. Assume that $\cA$ is $(\rho,\lambda)$-regular.

Define 
\begin{align}
\label{S5: def Ah}
\notag \Ah v
:=&\cA v - 2z\rho^{2-\lambda}\ev(\nabla v, \gd h) \\
& -[z \div(\rho^{2-\lambda}\gd h)+ z^2\rho^{2-\lambda} |\gd h|_g^2 -z\ev(\nabla h, a_1)]v .  
\end{align}
By \eqref{S5: a tech}, we thus have
$\Ah=e^{-zh}\circ \cA \circ e^{zh}$.

A function $h\in C^2(\M,\R)$ is said to belong to the class $\sH_\lambda(\M,g;\rho)$ with parameters $(c,M)$, if
\begin{itemize}
\item[($\sH_\lambda$1)] $M/c \leq \rho |\gd h|_g \leq Mc$;
\item[($\sH_\lambda$2)] $M/c \leq \rho^\lambda \div(\rho^{2-\lambda}\gd h) \leq Mc.$
\end{itemize}
Observe that if $h\in \sH_\lambda(\M,g;\rho)$ with parameters $(c,1)$, then $Mh\in \sH_\lambda(\M,g;\rho)$ with parameters $(c,M)$. 
\begin{definition}
A  singular manifold  $(\M,g;\rho)$ is said to enjoy {\em property $\sH_\lambda$}, if there exists some $h\in \sH_\lambda(\M,g;\rho)$. 
\end{definition}

We impose the following assumptions on the function $h$,  and the constant $z=a+ib$.
\begin{itemize}
\item[(H1)] $(\M,g;\rho)$ satisfies {\em property $\sH_\lambda$}, and $h\in \sH_\lambda(\M,g;\rho)$ with parameters $(c,M)$.
\item[(H2)] $\displaystyle a\in  (-\frac{1}{2Mc^3}, 0)$, and $|z|=1$.
\end{itemize}

Let $A_z:=-z \div(\rho^{2-\lambda}\gd h)- z^2\rho^{2-\lambda} |\gd h|_g^2 +z\ev(a_1,\nabla h) +a_0$.
By (H1), one can check that the operator $\Ah$ is $(\rho,\lambda)$-regular and $(\rho,\lambda)$-singular elliptic with $C_{\hat{\sigma}}=1$.
Moreover, (H1) implies
$$\rho^2|\gd h|^2_g/(Mc^3)\leq   \rho^\lambda |\div(\rho^{2-\lambda} \gd h)|  . $$
Lemma~\ref{S3: pts-mul}(c) yields
$$\rho^\lambda|\ev(\nabla h, a_1)|\leq  \rho|\gd h|_g \|a_1\|_{\infty;\lambda} .$$
Note that (H2) gives $b^2-a^2- \frac{a}{Mc^3}>1$.
We then have
\begin{align}
\label{S5: Ah-1}
\notag \rho^\lambda\Rp(A_z)
\notag&=\rho^2(b^2-a^2)|\gd h|^2_g-a\rho^\lambda\div(\rho^{2-\lambda} \gd h) + \rho^\lambda( a\ev(\nabla h, a_1)+ \Rp(a_0))\\
\notag&\geq  \rho^2 (b^2-a^2-\frac{a}{Mc^3} - \frac{c}{M}\|a_1\|_{\infty;\lambda})|\gd h|^2_g + \rho^\lambda\Rp(a_0) \\
& >C_0 \rho^2|\gd h|^2_g -\omega_\cA
\end{align}
for some $C_0>1$ and $\omega_\cA<0$ by choosing $M$  sufficiently large and the real part of $z$, i.e., $a$, satisfying (H2) accordingly. This shows that
$$
\rho|2z\gd h|_g=2\rho |\gd h|_g< \frac{2}{\sqrt{C_0}} \sqrt{\rho^\lambda\Rp(A_z)+\omega_\cA}.
$$
For any $\lambda^\prime\in\R$, let 
$$I(\lambda^\prime,\lambda,\tau,\sigma):=\{2\lambda^\prime+2\tau-2\sigma, 2\lambda^\prime-\lambda+2\tau-2\sigma\}.$$ 
By choosing $M$ large enough and making $z=a+ib$ satisfying (H2),  it holds that
\begin{align}
\label{S5: Ah-2}
\notag &\quad\rho^{\lambda-1}|-2z\rho^{2-\lambda}\gd h + t\rho^{2-\lambda}\gd\log\rho +a_1|_g\\
&< \frac{2}{\sqrt{C_1}} \sqrt{\rho^\lambda\Rp(A_z)+\omega_\cA}
\end{align}
for all $t\in I(\lambda^\prime,\lambda,\tau,\sigma)$ and some $\omega_\cA<0$, $C_1\in (1,C_0)$. Therefore, $\omega_\cA<0$ satisfies \eqref{S3: A3}-\eqref{S3: A4}.

We consider the following condition.
\begin{itemize}
\item[(H3)] $M$ is sufficiently large such that \eqref{S5: Ah-1} and \eqref{S5: Ah-2} hold.
\end{itemize}

Summarizing the above discussions, for $z=a+ib$ and $M$ satisfying (H2) and (H3), we conclude from Theorem~\ref{S3: domain} with $\omega=0$ that
\begin{align}
\label{S5: Ah-semigroup}
\Ah\in \cH(\mathring{W}^{2,\lambda^\prime -\lambda}_2(\M, V), L^{\lambda^\prime}_2(\M, V))\cap \Lis(\mathring{W}^{2,\lambda^\prime -\lambda}_2(\M, V), L^{\lambda^\prime}_2(\M, V)) , 
\end{align}
and the semigroup $\{e^{-t\Ah}\}_{t\geq 0}$ is contractive. 


For any function space $\F^{s,\vartheta}(\M, V)$ defined in Section~2,  the space 
$$e^{zh}\F^{s,\vartheta}(\M, V):=\{u\in L_{1,loc}(\M,V):\, e^{-zh}u\in  \F^{s,\vartheta}(\M, V)\}$$ 
is a Banach space equipped with the norm $\|\cdot \|_{e^{zh}\F^{s,\vartheta}}$, where
$$\|u \|_{e^{zh}\F^{s,\vartheta}}:= \|e^{-zh} u\|_{\F^{s,\vartheta}}. $$
It is easy to see that
\begin{align}
\label{S5.1: e^zh-Lis}
e^{zh}\in \Lis(e^{-zh}\F^{s,\vartheta}(\M, V), \F^{s,\vartheta}(\M, V)).
\end{align}
\begin{theorem}
\label{S5: A-semigroup}
Suppose that $(\M,g; \rho)$ is a singular manifold with  property $\sH_\lambda$, and $h\in \sH_\lambda(\M,g; \rho)$ with parameters $(c,M)$. 
Let $\lambda^\prime\in \R$, $\rho$ and $\lambda$ satisfy \eqref{S3: rho & lambda}.
Furthermore, assume  that the differential operator
\begin{align*}
\cA u:= -\div(\rho^{2-\lambda} \gd u)+\ev(\nabla u, a_1) +a_0 u
\end{align*}
is $(\rho,\lambda)$-regular.
Then, for any constant $z=a+ib$ and $M$ fulfilling (H2) and (H3), we have
$$\cA\in \cH(e^{zh}\mathring{W}^{2,\lambda^\prime -\lambda}_2(\M, V), e^{zh} L^{\lambda^\prime}_2(\M, V))\cap \Lis(e^{zh}\mathring{W}^{2,\lambda^\prime -\lambda}_2(\M, V), e^{zh} L^{\lambda^\prime}_2(\M, V)) ,$$
and the semigroup $\{e^{-t\cA}\}_{t\geq 0}$ is contractive.
\end{theorem}
\begin{proof}
Given any angle $\theta\in [0,\pi]$, set 
$$\Sigma_\theta:=\{z\in\C: |{\rm arg}z|\leq \theta \}\cup\{0\}.$$
\eqref{S5: Ah-semigroup} implies that
$S:=\Sigma_\theta \subset \rho(-\Ah)$ so that
$$|\mu|^{1-k} \| (\mu +\Ah)^{-1} \|_{\L(L^{\lambda^\prime}_2(\M,V), \mathring{W}^{2k,\lambda^\prime -k\lambda}_2(\M, V))} \leq \cE, \quad \mu\in S, \quad k=0,1, $$
for some $\theta\in [\pi/2,\pi)$ and $\cE>0$. By \eqref{S5.1: e^zh-Lis}
and $\cA= e^{zh} \circ \Ah \circ e^{-zh}$, it holds that $S \subset \rho(-\cA)$ and for all $\mu\in S$ and $k=0,1$
\begin{align*}
&\quad|\mu|^{1-k} \| (\mu +\cA)^{-1} \|_{e^{zh}\L(L^{\lambda^\prime}_2(\M,V),e^{zh} \mathring{W}^{2k,\lambda^\prime -k\lambda}_2(\M, V))} \\
&=|\mu|^{1-k} \| (\mu +e^{zh} \circ \Ah \circ e^{-zh})^{-1} \|_{\L(e^{zh} L^{\lambda^\prime}_2(\M,V), e^{zh}\mathring{W}^{2k,\lambda^\prime -k\lambda}_2(\M, V))} \leq \cE^\prime.
\end{align*}
Then the assertion follows from the well-known semigroup theory.
\end{proof}

\begin{remark}
Because the choice  of the constant $z$ and $M$ is not unique, it seems that the assertion in Theorem~\ref{S5: A-semigroup} is not well formulated. 
\smallskip\\
However, as is shown in Section~5.3 below, this is indeed not a problem.
In Theorem~\ref{S5: sing mfd-ends}, we will generalize the result in Theorem~\ref{S5: A-semigroup} to {\em singular manifolds} with $\sH_\lambda$-ends, which roughly speaking, means that a manifold satisfies {\em property $\sH_\lambda$} close to the singularities and is uniformly regular elsewhere. 
\smallskip\\
As  we will see in Theorem~\ref{S5: main theorem} and Corollary~\ref{S5: main corollary} below, for most of the practical examples, once an $h\in \sH_\lambda(\M,g;\rho)$ with parameters $(c,1)$ is fixed, we will see that the space $e^{zMh}\mathring{W}^{s,\vartheta}_p(\M, V)$ actually coincides with the weighted Sobolev-Slobodeckii space $\mathring{W}^{s,\vartheta+aM}_p(\M, V)$, for any $z=a+ib$ and $M$ fulfilling (H3) and (H4). 
\smallskip\\
Note that $aM\in (-\frac{1}{2c^3},0)$ in fact only depends on the constant $c$. 
Since the weight $\lambda^\prime$ is arbitrary, in Theorem~\ref{S5: A-semigroup}, we actually have that for any $\lambda^\prime$,
$$\cA\in \cH( \mathring{W}^{2,\lambda^\prime -\lambda}_2(\M, V),  L^{\lambda^\prime}_2(\M, V))\cap \Lis( \mathring{W}^{2,\lambda^\prime -\lambda}_2(\M, V),  L^{\lambda^\prime}_2(\M, V)) .$$
The result in Theorem~\ref{S5: A-semigroup} thus parallels to those in Section~3.
\end{remark}


\subsection{\bf Singular manifolds  with $\sH_\lambda$-ends}

\begin{definition}
An $m$-dimensional singular manifold $(\M,g; \rho)$ is called a  singular manifold with $\sH_\lambda$-ends if it 
satisfies the following conditions.
\begin{itemize}
\item[(i)] ${\bf G}=\{G_1,\cdots,G_n\}$ is a finite set of disjoint closed subsets of $\M$. Each $(G_i,g; \rho_i)$ is an $m$-dimensional  singular manifold satisfying  property $\sH_\lambda$.
\item[(ii)] $G_0$ is closed in $\M$, and  $(G_0,g)$ is an $m$-dimensional  uniformly regular Riemannian manifold.
\item[(iii)] $\{G_0\}\cup {\bf G}$ forms a covering for $\M$. $\partial_0 G_i:= G_0\cap G_i \subset \partial G_0 \cap \partial G_i$.
\item[(vi)] Let $\rho_i:=\rho|_{G_i}$. Either of the following conditions holds true
$$\rho_i \leq 1,\quad  i=1,\cdots,n;\quad  \text{ or }\quad  \rho_i \geq 1,\quad  i=1,\cdots,n.$$
\end{itemize}
$G_i$ are called the  $\sH_\lambda$-ends of $\M$.
\end{definition}

In the following, we will present several examples of {\em singular manifolds}  with  $\sH_\lambda$-ends, and show how to construct such manifolds in a systematic way.

The proof for the following lemma is  straightforward.
\begin{lem}
\label{S5: lem-product}
Suppose that $(\M,g;\rho)$ has {\em property $\sH_\lambda$}, $h\in \mathscr{H}(\M,g;\rho)$ with parameter $(c,M)$, and $(B,g_B)$ is a  uniformly regular Riemannian manifold. Then $(\M\times B, g + g_B ;\rho\otimes {\bf 1}_B)$ also has {\em property $\sH_\lambda$}, and
$$h\otimes{\bf 1}_B  \in \sH_\lambda(\M\times B, g + g_B ;\rho\otimes {\bf 1}_B) $$
with parameter $(c,M)$. 
\end{lem}

\begin{lem}
\label{S5: lem-diff}
Let $f: \tilde{\M} \rightarrow \M$ be a diffeomorphism of manifolds.
Suppose that $(\M,g;\rho)$ has {\em property $\sH_\lambda$}, and $h\in \sH(\M,g;\rho)$ with parameters $(c,M)$.
\smallskip\\
Then so does $(\tilde{\M}, f^* g; f^*\rho)$, and $f^* h\in \sH(\tilde{\M}, f^* g; f^*\rho)$ with parameters $(c,M)$.
\end{lem}
\begin{proof}
It is a simple matter to check that $(f^{-1}(\Ok), f^*\vpk)_{\kappa\in\K}$ forms a uniformly regular atlas for $\tilde{\M}$ and
$$(f^*\vpk)_* f^*h= \kf h, \quad (f^*\vpk)_*(f^*g)=\kf g. $$
As a direct consequence, we have the identities 
$$ (f^*\vpk)_* \gd_{f^*g} f^*h= \kf \gd_g h,$$
and
$$(f^*\vpk)_* \div_{f^*g}((f^*\rho)^{2-\lambda} \gd_{f^*g} f^*h )= \div_g(\rho^{2-\lambda} \gd_g h). $$
\end{proof}

The following examples show that we can construct a family of {\em singular manifolds} with $\sH_\lambda$-ends in a great variety of geometric constellations. 
In particular, we can find manifolds with $\sH_\lambda$-type singularities of arbitrarily high dimension.

Let $J_0:=(0,1]$ as in Section~4. 
We will introduce some subsets of the class $\mathscr{C}(J_0)$, which is very useful for constructing examples of {\em singular manifolds} with $\sH_\lambda$-ends.
We call a {\em cusp characteristic} $R\in \mathscr{C}(J_0)$ a {\em mild cusp characteristic} if $R$ satisfies \eqref{S4: cusp ch} and \eqref{S5: mild cusp ch} below.
\begin{align}
\label{S5: mild cusp ch}
\dot{R} \sim {\bf 1}_{J_0}.
\end{align}
If $R$ further satisfies
\begin{align}
\label{S5: unif mild cusp ch}
|\ddot{R}| <\infty,
\end{align}
then we call it a {\em uniformly mild cusp characteristic}.
We write $R\in \mathscr{C_U}(J_0)$.
\begin{example}
$R(t)=t$, $R(t)=\frac{4}{\pi}\arctan t$, $R(t)=\log(1+(e-1)t)$, $R(t)=2t/3+\sin(\frac{\pi}{2}t)/3$ are examples of  uniformly mild cusp characteristics.
\end{example} 
\begin{lem}
\label{S5: Eg-mild cusp ch}
Suppose that $R\in \mathscr{C_U}(J_0) $ and $\lambda\in [0,1)\cup(1,\infty)$. Then $(J_0,dt^2; R)$ is a   singular manifold  with  $\sH_\lambda$-end.
\end{lem}
\begin{proof}
First, by \cite[Lemma~5.2]{Ama14}, $(J_0,dt^2; R)$ is a {\em singular manifold}. We set 
\begin{align}
\label{S5: h-unif mild cusp}
h(t)=\sg(1-\lambda)\log R(t).
\end{align}
Then $R(t)|\dot{h}(t)|=\dot{R}(t) \sim {\bf 1}_{J_0}$ on $ J_0$, and 
$$R^\lambda(t)\frac{d}{dt} (R^{2-\lambda}(t)\dot{h}(t))=|1-\lambda||\dot{R}(t)|^2 +\sg(1-\lambda) R(t) \ddot{R}(t) \sim {\bf 1}_{I_c},  $$
where $I_c:=(0,c]$ for $c$ small enough. Then the assertion follows from the fact that $([c,1], dt^2)$ is uniformly regular for any $c>0$.
\end{proof}
\begin{remark}
We can actually show that $(J_0,dt^2; R)$ is a {\em singular manifold} with {\em property $\sH_\lambda$} with 
$$h(t):=\sg(\lambda-1)\int\limits_t^1 ds/R(s)\in \sH_\lambda(J_0,dt^2;R), $$
as long as $R$ is a {\em mild cusp characteristic}.
But for the sake of practical usage, we will see in Section~5.3 below that \eqref{S5: h-unif mild cusp} benefits us more in establishing the correspondence of the space $e^{zh}\mathring{W}^{s,\vartheta}_p(\M,V)$ with weighted Sobolev-Slobodeckii spaces.
\end{remark}

Suppose that $R\in \mathscr{C_U}(J_0) $, $(B,g_B)$ is a uniformly regular Riemannian submanifold of $\R^{d-1}$, and $(\Gamma, g_\Gamma)$ is a compact connected Riemannian manifold without boundary. 
We call $(\M,g)$ a {\em uniformly mild $\Gamma$-wedge over $P(R,B)$},  if there is a diffeomorphism $f: \M \to W(R,B,\Gamma)$ such that $g=f^*( \phi_P^*(dt^2+g_B) +g_\Gamma)$.
\begin{prop}
\label{S5: Eg-mild cusp-wedge}
Let $\lambda\in [0,1)\cup(1,\infty)$. Assume that $(\M,g)$ is a uniformly mild $\Gamma$-wedge over $P(R,B)$. 
Then $(\M,g)$ is a singular manifold  with $\sH_\lambda$-end.
\end{prop}
\begin{proof}
Lemma~\ref{S4: wedge} implies that 
$$(\M,g; f^*(\phi_P^*(R\otimes{\bf 1}_B)\otimes {\bf 1}_\Gamma))$$
is a {\em singular manifold}.
We define 
$$h(t):=\sg(1-\lambda)\log R(t).$$
Put $I_c:=(0,c]$ and $\M_c:=f^{-1}(P(R|_{I_c},B)\times\Gamma)$.
It follows from Lemmas~\ref{S4: lem 4.1}, \ref{S5: lem-product}, \ref{S5: lem-diff}, and~\ref{S5: Eg-mild cusp ch} that for $c>0$ sufficiently small,  $(M_c,g)$ has {\em property $\sH_\lambda$} with 
$$f^*(\phi_P^*(h\otimes{\bf 1}_B)\otimes{\bf 1}_\Gamma)  \in \sH_\lambda(\M_c,g; f^*(\phi_P^*(R|_{I_c}\otimes{\bf 1}_B)\otimes {\bf 1}_\Gamma)).$$
\end{proof}
\begin{remark}
As before, in fact, we only need to require $R$ to be a {\em mild cusp characteristic}.
Let $h(t):=\sg(\lambda-1)\int\limits_t^1 ds/R(s).$ 
Then
$(\M,g)$ has {\em property $\sH_\lambda$} with 
$$f^*(\phi_P^*(h\otimes{\bf 1}_B)\otimes{\bf 1}_\Gamma)  \in \sH_\lambda(\M,g; f^*(\phi_P^*(R\otimes{\bf 1}_B)\otimes {\bf 1}_\Gamma)).$$
\end{remark}
In the following examples, we always assume that $(B,g_B)$ is a compact closed $C^\infty$-Riemannian manifold.
\begin{example}
By the above proposition, we can easily verify that the following manifolds enjoy  property $\sH_\lambda$.
\begin{itemize}
\item[(a)] Suppose $\M$ is a cone, i.e., $\M=([0,1] \times B)/( \{0\}\times B)$.  We equip $\M$ with the conventional metric $g=dt^2 + t^2 g_B$. Then $(\M,g;\phi_P^*(t\otimes {\bf 1}_B))$ enjoys  property $\sH_\lambda$  for $\lambda\in [0,1)\cup(1,\infty)$. 
\item[(b)] Suppose $\M$ is an edge manifold, that is, $(\M,g;\rho)=(P(t,B)\times\R^d, \phi_P^*(dt^2 + g_B)+g_d;\phi_P^*(t\otimes {\bf 1}_B)\otimes {\bf 1}_{\R^d})$. Then $(\M,g;\rho)$ enjoys property $\sH_\lambda$ for $\lambda\in [0,1)\cup(1,\infty)$. 
\end{itemize}
\end{example}
In some references, the authors equip an edge with the metric $g=dt^2/t^2 + g_B +g_\Gamma/t^2$, which makes $(\M,g)$ uniformly regular. This case has been studied in depth in \cite{Ama13b}.

Given any compact submanifold $\Sigma\subset (\M,g)$, the distance function 
is a well-defined smooth function in a collar neighborhood $\mathscr{U}_\Sigma$ of $\Sigma$. The distance ball at $\Sigma$ with radius $r$ is defined by 
$$\B_\M(\Sigma,r):= \{\p\in \M: {\rm dist}_\M(\p,\Sigma)<r \}. $$
\begin{prop}
\label{S5: sing mdf with holes}
Suppose that $(\M,g)$ is a singular manifold with holes.
More precisely, $(\sM,g)$ is a uniformly regular Riemannian manifold. 
${\bf \Sigma}=\{\Sigma_1,\cdots,\Sigma_k\}$ is a finite set of disjoint $m$-dimensional compact manifolds with boundary such that $\Sigma_j\subset \mathring{\sM}$. 
Let $\lambda\in [0,1)\cup(1,\infty)$.
Put $\M:=\sM\setminus \cup_{j=1}^k \Sigma_j$ and 
$$\mathscr{B}_{j,r}:= \bar{\B}_\sM(\partial\Sigma_j,r)\cap \M,\quad j=1,\cdots,k.$$ 
Then we can find a singularity function $\rho$ satisfying
$$\rho|_{\mathscr{B}_{j,r}}=:\rho_j= {\rm dist}_{\sM}(\cdot, \partial \Sigma_j), $$
for some $r\in [0,\delta)$, where $\delta<{\rm diam}(\sM)$ fulfils that $\mathscr{B}_{i,\delta}\cap \mathscr{B}_{j,\delta} =\emptyset$ for $i\neq j$, and
$$\rho\sim {\bf 1},\quad \text{elsewhere on }\M. $$
Moreover, 
$(\M,g;\rho)$
is a singular manifold with $\sH_\lambda$-ends.
\end{prop}
\begin{proof}
By Lemma~\ref{S4: holes},  $(\M,g)$ is a {\em singular manifold}.
We will show that $\rho_j:={\sf dist}_{\sM}(\cdot,\partial\Sigma_j)$ is a singularity function for $\mathscr{B}_{j,r}$ and  
$$h_j:= \sg(1-\lambda)\log \rho_j \in \sH_\lambda(\mathscr{B}_{j,r},g;\rho_j)  $$
for sufficiently small $r$.
By the collar neighborhood theorem, there exists an open neighborhood $\mathscr{V}_{j,\varepsilon}$ of $\partial \Sigma_j$ in the closure of $\M$ in $\sM$, i.e., $\bar{\M}$, and a diffeomorphism $f_j$ such that 
$$f_j:  \mathscr{V}_{j,\varepsilon} \to \partial \Sigma_j \times [0,\varepsilon), \quad (f_j)_* g|_{\mathscr{V}_{j,\varepsilon}}= g|_{\partial \Sigma_j} + dt^2 , $$
for some $\varepsilon>0$. 
Note that $\rho_j$ is a well defined smooth function in $\mathscr{V}_{j,\varepsilon}$ for $\varepsilon$ sufficiently small.
Let $T^\perp \partial \Sigma_j$ denote the normal bundle of $\partial \Sigma_j$ in $\bar{\M}$. At every point $\p\in \partial\Sigma_j$, there exists a unique $\nu_\p\in T^\perp_\p \partial \Sigma_j$ such that
$$ T_\p f_j \nu_\p=e_1\in T_0 \R. $$
Then, $f_j^{-1}(\p,t)=exp_\p(t\nu_\p)$, where $exp_\p$ is the exponential map at $\p$.
Therefore,
$$(f_j)_* \rho_j (\p,t):= t \beta_j(\p),\quad  \text{in }  \partial\Sigma_j \times [0,\varepsilon),$$
for some $\beta_j \in C^\infty (\partial\Sigma_j)$ and $\beta_j \sim {\bf 1}_{\partial\Sigma_j}$. Because of the compactness of $\partial\Sigma_j$, by choosing $\varepsilon$ small enough, we can easily show that
$$|\nabla \rho_j |_g \sim {\bf 1}_{\mathscr{V}_{j,\varepsilon}},\quad |\Delta \rho_j|<\infty ,\quad  \text{in }\mathscr{V}_{j,\varepsilon}. $$
Here $\Delta$ is the Laplace-Beltrami operator with respect to the metric $g$ defined by
$\Delta=\Delta_g:=\div\circ\gd$.
Since $\mathscr{B}_{j,r}\subset \mathscr{V}_{j,\varepsilon}$ for $r$ small enough, in view of
$$\rho_j\gd h_j =\sg(1-\lambda) \gd \rho_j,$$
and
$$
\rho_j^\lambda \div(\rho^{2-\lambda}\gd h_j)=\sg(1-\lambda)\rho_j \Delta \rho_j + |1-\lambda| |\gd \rho_j|_g^2,  $$
we immediately conclude that $h_j$ satisfies ($\sH_\lambda$1) and ($\sH_\lambda$2) in $\mathscr{B}_{j,r}$ for $r$ small enough. 

Because $|\nabla \rho_j |_g \sim {\bf 1}$ in $\mathscr{B}_{j,r}$ for $r$ small enough, we can infer from the implicit function theorem that 
$$S_{j,r_0}:=\{\p\in\M: {\sf dist}_{\sM}(\p, \partial\Sigma_j)=r_0\}\cap \M$$ 
is a compact submanifold for some $r_0\in (0,r)$. 
By the tubular neighborhood theorem, we can easily show that $(\mathscr{B}_{j,r_0},g)$ and 
$(\M\setminus \cup_{j=1}^k \mathring{\mathscr{B}}_{j,r_0},g)$ are all manifolds with boundary.

By \cite[Corollary~4.3]{Ama14}, $(\partial\Sigma_j, g|_{\partial\Sigma_j})$ is uniformly regular. 
In particular, taking $\beta_j$ as a singularity function, $(\partial\Sigma_j, g|_{\partial\Sigma_j}; \beta_j)$ can be considered as a {\em singular manifold}. By Lemmas~\ref{S5: lem-product} and \ref{S5: lem-diff},   we conclude that for $r$ sufficiently small
$(\mathscr{B}_{j,r_0},g; \rho_j) $ is a {\em singular manifold} with boundary $S_{j,r_0}$.

Based on the collar neighborhood theorem, we can find an open neighborhood $\mathscr{U}_{j,\varepsilon}\subset \mathscr{B}_{j,r}$ of $S_{j,r_0}$ in $\M\setminus \cup_{j=1}^k \mathring{\mathscr{B}}_{j,r_0}$ such that there is a diffeomorphism 
$$\phi_j: \mathscr{U}_{j,\varepsilon} \to S_{j,r_0} \times [0,\varepsilon), \quad \phi_j^* g|_{\mathscr{U}_{j,\varepsilon}}= g|_{S_{j,r_0}} + dt^2 ,$$
with $\phi_j (S_{j,r_0})=S_{j,r_0} \times \{0\}$.
We choose a function $\xi \in BC^\infty([0,\varepsilon), [0,1])$ such that
$$\xi|_{[0,\varepsilon/4]}\equiv 0, \quad \xi|_{[\varepsilon/2,\varepsilon)} \equiv 1.$$
Put $\xi_{j,0}:=\phi_j^*({\bf 1}_{S_{j,r_0}}\otimes \xi)$. 
Similarly, we can find  $\xi_{j,j}\in BC^\infty(\mathscr{U}_{j,\varepsilon},[0,1])$ such that
$$\xi_{j,j}|_{\phi_j^{-1}(S_{j,r_0} \times [0,\varepsilon/2])}\equiv 1,\quad \xi_{j,j}|_{\phi_j^{-1}(S_{j,r_0} \times[3\varepsilon/4,\varepsilon))}\equiv 0. $$
We define $\xi_i \in C^\infty(\M,[0,1])$ with $i=0,\cdots,k$  as follows. For $j=1,\cdots,k$,
\begin{align*}
\xi_j(\p) =
\begin{cases}
1, \quad \quad & \p \in \mathscr{B}_{j,r_0},\\
 \xi_{j,j}, &\p\in \mathscr{U}_{j,\varepsilon},\\
0, & \text{elsewhere},
\end{cases}
\quad \text{ and } \quad
\xi_0(\p) =
\begin{cases}
0, \quad \quad & \p \in \mathscr{B}_{j,r_0},\\
 \xi_{j,0}, &\p\in \mathscr{U}_{j,\varepsilon},\\
1, & \text{elsewhere}.
\end{cases}
\end{align*}
Put 
$\rho: = \xi_0 {\bf 1}_\M+ \sum_{j=1}^k \xi_{j}\rho_j. $ Then it is not hard to see that $\rho$ is a singularity function for $(\M,g)$ such that $\rho\sim {\bf 1}$ on  $\M\setminus \cup_{j=1}^k \mathring{\mathscr{B}}_{j,r_0}$ and $\rho|_{\mathscr{B}_{j,r_0}}= \rho_j$. Therefore, $(\M\setminus \cup_{j=1}^k \mathring{\mathscr{B}}_{j,r_0},g)$ is a {\em uniformly regular Riemannian manifold}.

Summarizing the above discussions, we have proved that $(\M,g;\rho)$ is a {\em singular manifold} with $\sH_\lambda$-ends.
\end{proof}
From the above proof, it is easy to see that the following corollary holds.
\begin{cor}
\label{S5: mfd with cpt bdry}
Suppose that $(\sM,g)$ is a uniformly regular Riemannian manifold  with compact boundary.
Let $\lambda\in [0,1)\cup(1,\infty)$. 
Put $\M:=\mathring{\sM}$ and 
$$\mathscr{B}_r:= \bar{\B}_\sM(\partial\sM,r)\cap \M,\quad j=1,\cdots,k.$$ 
Then there exists a singularity function $\rho$ satisfying
$$\rho|_{\mathscr{B}_r}=:\rho_j= {\rm dist}_{\sM}(\cdot, \partial \sM), $$
for some $r>0$,  and
$$\rho\sim {\bf 1},\quad \text{elsewhere on }\M. $$
Moreover,
$(\M,g;\rho)$
is a singular manifold with $\sH_\lambda$-ends.
\end{cor}
\begin{remark}
\label{S5.2: remove bdry}
More generally, we can take $\boldsymbol{\Sigma}=\{\Sigma_1,\cdots,\Sigma_k\}$ to be a finite set of disjoint compact closed submanifolds of codimension at least $1$ such that $\Sigma_j \subset \partial\mathscr{M}$ if $\Sigma_j \cap \partial\mathscr{M} \neq \emptyset.$ In \cite[Theorem~1.6]{Ama14}, it is shown that $\M:=\mathscr{M}\setminus \cup_{j=1}^k \Sigma_j$ is a {\em singular manifold}. Indeed, we can prove that this is a {\em singular manifold} with $\sH_\lambda$-ends. The proof is quite similar to that for Proposition~\ref{S5: sing mdf with holes}, but more technical. To keep this article at a reasonable length, we will not present a proof herein.
\end{remark}
\begin{remark}
\label{S5.2: remove pts}
In Proposition~\ref{S5: sing mdf with holes}, we can also allow $\boldsymbol{\Sigma}=\{\p_1,\cdots,\p_k\}$ to be a finite set of discrete points in $\mathring{\sM}$. Then 
$$(\M,g;\rho):= (\sM\setminus \cup_{i=1}^k \bar{\B}_{\sM}(\p_i,r), g;\rho) $$
is still a {\em singular manifold}. Here $\rho$ is defined in the same way as in Proposition~\ref{S5: sing mdf with holes}.
\smallskip\\
An estimate for $\Delta \rho_j$ can be obtained from the fact that for $r$ sufficiently small
$$\Delta \rho_j(\p)= \frac{m-1}{\rho_j(\p)} + O(\rho_j(\p)),\quad \text{in } \mathscr{B}_{j,r}.$$
See \cite[formulas~(1.134), (1,159)]{ChowLuNi06}. 
Taking $h_j=\log \rho_j$, we have
\begin{align*}
\rho_j^\lambda \div (\rho_j^{2-\lambda} \gd h_j)= \rho_j \Delta \rho_j + (1-\lambda) | \gd \rho_j|_g^2 =m-\lambda +O(\rho_j^2),
\end{align*}
since $|\gd \rho_j|_g=1$. We immediately have 
$$\sg(m-\lambda) h_j \in \sH_\lambda(\mathscr{B}_{j,r},g; \rho) $$
for sufficiently small $r$ and $\lambda\geq 0$ with $\lambda\neq m$. 
\smallskip\\
Therefore, $(\M,g;\rho)$ is indeed a {\em singular manifold} with $\sH_\lambda$-ends.
\end{remark}


\subsection{\bf $L_p$-theory on singular manifolds with $\sH_\lambda$-ends}

\begin{theorem}
\label{S5: sing mfd-ends}
Suppose that $(\M,g; \rho)$ is a singular manifold with $\sH_\lambda$-ends. 
Let $\lambda^\prime\in \R$, $\rho$ and $\lambda$ satisfy  \eqref{S3: rho & lambda}.
Furthermore, assume that the differential operator
\begin{align*}
\cA u:= -\div(\rho^{2-\lambda} \gd u)+\ev(\nabla u, a_1) +a_0 u
\end{align*}
is $(\rho,\lambda)$-regular.
Then, for any constant $z=a+ib$ and $M$ satisfying (H2) and (H3) on all the $\sH_\lambda$-ends $G_i$ with $i=1,\cdots,n$, we have
$$\cA\in \cH(e^{zh}\mathring{W}^{2,\lambda^\prime -\lambda}_2(\M, V), e^{zh} L^{\lambda^\prime}_2(\M,V)).$$
\end{theorem}
\begin{proof}
Without loss of generality, we may assume that $\partial_0 G_i\neq \emptyset$ for $i=1,\cdots,n$.
It is not hard to see that $\partial_0 G_i$ is a component of $\partial G_i$. 

(i) Based on the collar neighborhood theorem, we can find an open neighborhood $U_i$ of $\partial_0 G_i$ in $G_i$ such that there is a diffeomorphism 
$$\phi_i: U_i \to \partial_0 G_i \times [0,1), \quad \phi_i^* g|_{U_i}= g|_{\partial_0 G_i} + dt^2 ,$$
with $\phi_i (\partial_0 G_i)=\partial_0 G_i \times \{0\}$,
and
$$\rho_i|_{U_i}\sim {\bf 1}_{U_i},\quad i=1,\cdots,n. $$
We choose functions $\xi, \tilde{\xi} \in BC^\infty([0,1), [0,1])$ such that
$$\xi|_{[0,1/2]}\equiv 1, \quad \xi|_{[3/4,1)} \equiv 0; \quad \tilde{\xi}|_{[0,1/4]}\equiv 0, \quad \tilde{\xi}|_{[1/2,1)} \equiv 1.$$
Set $\hat{\pi}_{i,0}:=\phi_i^*({\bf 1}_{\partial_0 G_i}\otimes \xi)$ and  $\hat{\pi}_{i,i}:=\phi_i^*({\bf 1}_{\partial_0 G_i}\otimes \tilde{\xi})$.
We define $\tilde{\pi}_j \in C^\infty(\M,[0,1])$ with $j=0,\cdots,n$  as follows. For $i=1,\cdots,n$,
\begin{align*}
\tilde{\pi}_i(\p) =
\begin{cases}
1, \quad \quad & \p \in G_i\setminus U_i,\\
 \hat{\pi}_{i,i}, &\p\in U_i,\\
0, & \text{elsewhere},
\end{cases}
\quad \text{ and } \quad
\tilde{\pi}_0(\p) =
\begin{cases}
1, \quad \quad & \p \in G_0,\\
 \hat{\pi}_{i,0}, &\p\in U_i,\\
0, & \text{elsewhere}.
\end{cases}
\end{align*}
For $j=0,\cdots,n$, we set
$$\pi_j = \frac{\tilde{\pi}_j}{\sqrt{\sum_{i=0}^n \tilde{\pi}_i^2}}. $$
Then $(\pi_j^2)_{j=0}^n$ forms a partition of unity on $\M$, and $\pi_j\in BC^{\infty,0}(\M).$ 

Put $\hat{G}_0:=G_0\cup\bigcup_{i=1}^n \bar{U}_i$, which is uniformly regular. 
Define 
$$\mathring{\bW}_2^{s,\vartheta}(\M,V):=\prod_{j=0}^n \mathring{W}_2^{s,\vartheta}(X_j,V),$$ where $X_j:=G_j$ for $j=1,\cdots,n$, and $X_0:=\hat{G}_0$.
It is understood that on $X_0$, the singularity function can be taken as ${\bf 1}_{X_0}$, and thus the definition of weighted function spaces on $X_0$ is independent of the choice of the weight $\vartheta$.
We further introduce two maps:
$$\Lambda^c: \mathring{W}_2^{s,\vartheta}(\M,V)\to  \mathring{\bW}_2^{s,\vartheta}(\M,V):\quad u\mapsto (\pi_j u)_{j=0}^n,$$
and 
$$\Lambda: \mathring{\bW}_2^{s,\vartheta}(\M,V)\to  \mathring{W}_2^{s,\vartheta}(\M,V):\quad ( u_j )_{j=0}^n\mapsto \sum_{j=0}^n \pi_j u_j.$$
By Proposition~\ref{S2: pointwise multiplication}, we immediately conclude that $\Lambda$ is a retraction from the space $\mathring{\bW}_2^{s,\vartheta}(\M,V)$ to $\mathring{W}_2^{s,\vartheta}(\M,V)$ with $\Lambda^c$ as a coretraction.

(ii) We  show that  there exists some $h\in C^2(\M)$ such that $h_i:=h|_{G_i}\in \sH_\lambda(G_i,g; \rho_i)$ with uniform parameters $(c,M)$ for $i=1,\cdots,n$, and $h_0:=h|_{G_0}\in BC^2(G_0)$. 

Since $G_i$ has {\em property $\sH_\lambda$}, we can find $h_i\in \sH_\lambda(G_i,g; \rho_i)$ with uniform parameters $(c,M)$ on all $\sH_\lambda$-ends $G_i$ for $i=1,\cdots,n$. 
Note that for $u\in C^2(\M)$, it follows from \cite[formula~A.9]{Ama13b} and \eqref{S2: nabla-grad} that
\begin{align*}
|\Delta u|= |C^{\sigma+1}_{\tau+1} \nabla \gd u| =| \nabla \gd u|_g= | \gd^2 u|_g= |\nabla^2 u|_g. 
\end{align*}
Therefore, ($\sH_\lambda$1) and ($\sH_\lambda$2) actually imply that $h_i \in BC^{2,0}(G_i)$.

Since $\partial_0 G_i$ is a compact submanifold of $\M$, by the tubular neighborhood theorem, we can find an closed neighborhood $\tilde{U}_i$ of $\partial_0 G_i$ in $\M$ such that $\tilde{U}_i\cap G_j=\emptyset$ for $j\neq 0, i$, and  there is a diffeomorphism 
$$\tilde{\phi}_i: \tilde{U}_i \to \partial_0 G_i \times [-1,1], \quad \tilde{\phi}_i^* g|_{\tilde{U}_i}= g|_{\partial_0 G_i} + dt^2 ,$$
with the convention $\tilde{\phi}_i:\tilde{U}_i \cap G_i\to\partial_0 G_i \times [0,1)$,
and
$$\rho_i|_{\tilde{U}_i}\sim {\bf 1}_{\tilde{U}_i},\quad i=1,\cdots,n. $$
By a similar construction as in Step (i), we can find $\xi, \tilde{\xi}\in BC^\infty([-1,1],[0,1])$ with
$$\xi |_{[-1, -1/2]}\equiv 1,\quad \xi|_{[-1/4,1]}\equiv 0;\quad \tilde{\xi}|_{[-1, -3/4]}\equiv 0,\quad \tilde{\xi}|_{[-1/2,1]}\equiv 1. $$
Set $\xi_{i,0}:= \tilde{\phi}_i^*({\bf 1}_{\partial_0 G_i} \otimes \xi)$ and $\xi_{i,i}:= \tilde{\phi}_i^*({\bf 1}_{\partial_0 G_i} \otimes \tilde{\xi})$. Then we define
\begin{align*}
\xi_i:= 
\begin{cases}
\xi_{i,i},\quad &\text{on } \tilde{U}_i;\\
1, &\text{on } G_i\setminus\tilde{U}_i;\\
0, &\text{elsewhere},
\end{cases}
\quad \text{and}\quad
\xi_0:= 
\begin{cases}
\xi_{i,0},\quad &\text{on } \tilde{U}_i;\\
0, &\text{on } G_i\setminus\tilde{U}_i;\\
1, &\text{elsewhere}.
\end{cases}
\end{align*}
The compactness of $\partial_0 G_i$ and \cite[Corollary~4.3]{Ama14} imply that $\partial_0 G_i$ is uniformly regular. Therefore, we find for $\partial_0 G_i$  a uniformly regular atlas $\hat{\A}_i:=(\hat{\sf O}_{\kappa,i}, \hat{\varphi}_{\kappa,i})_{\kappa\in \frak{K}_i}$, 
and a localization system $(\hat{\pi}_{\kappa,i})_{\kappa\in \frak{K}_i}$.
We set 
$${\sf O}_{\kappa,i} = \tilde{\phi}_i^{-1}(\hat{\sf O}_{\kappa,i} \times [-1,1]),\quad \varphi_{\kappa,i}=(\hat{\varphi}_{\kappa,i},\id)\circ \tilde{\phi}_i, $$
and $\pi_{\kappa,i}:=\tilde{\phi}_i^*(\hat{\pi}_{\kappa,i} \otimes {\bf 1}_{[-1,1]})$. Then $(\pi_{\kappa,i}^2)_{\kappa\in\frak{K}_i}$ forms a partition of unity on $ \tilde{U}_i$.

Let $\psi_{\kappa,i}= [\varphi_{\kappa,i}]^{-1}$.
We define
$$\mathcal{R}_i^c: BC^k(\M_i)\to \boldsymbol{BC}^k(\mathbb{U}),\quad u \mapsto (\psi_{\kappa,i}^*(\pi_{\kappa,i}u))_{\kappa\in \frak{K}_i} ,$$
and
$$\mathcal{R}_i:  \boldsymbol{BC}^k(\mathbb{U})\to BC^k(\M_i) ,\quad (u_\kappa)_{\kappa\in \frak{K}_i} \mapsto \sum_{\kappa\in \frak{K}_i}\psi_{\kappa,i}^*(\pi_{\kappa,i}u_\kappa) .$$
Here $\boldsymbol{BC}^k(\mathbb{U}):= \prod_{\kappa\in\frak{K}_i} BC^k(\mathbb{U}_\kappa)$ and
\begin{align*}
\mathbb{U}_\kappa=
\begin{cases}
\R^{m-1}\times [-1,1], \quad & \text{if }\M_i=\tilde{U}_i;\\
\R^{m-1}\times [0,1], & \text{if }\M_i=\tilde{U}_i\cap G_i.
\end{cases}
\end{align*}
Then alike to Proposition~\ref{S2: retraction&coretraction}, we can show that $\mathcal{R}_i$ is a retraction from $\boldsymbol{BC}^k(\mathbb{U})$ to $BC^k(\M_i)$ with $\mathcal{R}_i^c$ as a coretraction. 

By a well-known extension theorem, there exists a universal extension operator 
$$ \fE\in \L( BC^k(\R^{m-1}\times [0,1]), BC^k(\R^{m-1}\times [-1,1])) .$$ 
Set $\boldsymbol{\fE}\in \L(\boldsymbol{BC}^k(\R^{m-1}\times [0,1]), \boldsymbol{BC}^k(\R^{m-1}\times [-1,1]))$ and
$$\fE_i:=\mathcal{R}_i\circ  \boldsymbol{\fE}\circ \mathcal{R}_i^c,\quad i=1,\cdots,n.$$
Note that $(G_i \cup\tilde{U}_i ,g ;\rho)$ is a {\em singular manifold}.
Then 
$$ \fE_i \in \L(BC^{k,0}(G_i), BC^{k,0}(G_i\cup \tilde{U}_i)), \quad i=1,\cdots,n,\quad k\in \Nz.$$
Here we adopt the convention that $\fE_i u(\p)=u(\p)$ for any point $\p \in G_i\setminus \tilde{U}_i$.
Put $\tilde{h}_i:= \fE_i h_i$. We thus have $\tilde{h}_i\in BC^{2,0}(G_i\cup \tilde{U}_i)$. 
Now we define 
$$h= \xi_0 {\bf 1}_{\M} + \sum_{i=1}^n\xi_i\tilde{h}_i. $$
Then $h \in C^2(\M)$ satisfies the desired properties.

(iii) One can verify that for $j=0,\cdots,n$ and any $v\in \mathcal{D}(\M,V)$
\begin{align}
\label{S5: commutator}
\notag \pi_j \Ah v=& \Ah (\pi_j v)  + 2\rho^{2-\lambda}\ev(\nabla v, \gd \pi_j)\\
\notag& +  [\div(\rho^{2-\lambda}\gd \pi_j)- \ev(\nabla \pi_j, a_1)+ 2z\rho^{2-\lambda}\ev(\nabla \pi_j, \gd h) ] v\\
=&:\Ah (\pi_j v) +\cB_j v ,
\end{align}
where the operator $\Ah$ is defined in \eqref{S5: def Ah}.
Note that $\rho|_{\cup_{j=0}^n \supp(|\nabla \pi_j|_g)}\sim {\bf 1},$ and thus 
$$\gd \pi_j \in BC^{\infty,\vartheta}(\M, T\M)$$ 
for any $\vartheta\in\R$.
Based on these observations and  Propositions~\ref{S2: pointwise multiplication}, \ref{S2: nabla}, \ref{S2: div-tangent},  and \cite[Corollaries~7.2, 12.2]{Ama13}, we infer that
\begin{align}
\label{S5.3: Besov-Lp}
\cB_j \in \L(\mathring{B}_2^{1,\lambda^\prime-\lambda/2}(\M,V), L_2^{\lambda^\prime}(X_j,V)), \quad j=0,\cdots,n.
\end{align}
Set $\cA_{h_j}:=\Ah|_{X_j}$.
\eqref{S5: Ah-semigroup} and \cite[Theorem~5.2]{Ama13b} yield
$$\cA_{h_j}\in \cH(\mathring{W}^{\lambda^\prime -\lambda}_2(X_j, V), L_2^{\lambda^\prime}(X_j, V)),\quad j=0,\cdots,n. $$ 
Put $\bar{\cA}_h:=(\cA_{h_j})_{j=0}^n$ and
$$E_1:=\mathring{\bW}^{2,\lambda^\prime-\lambda}_2(\M,V),\quad E_0:=\boldsymbol{L}_2^{\lambda^\prime}(\M,V). $$
Then there exist some $\theta\in [\pi/2,\pi)$, $\omega_0\geq 0$ and $\cE>0$ such that $S_0:=\omega_0+\Sigma_\theta \subset \rho(-\bar{\cA}_h)$ and
$$|\mu|^{1-k} \| (\mu+ \bar{\cA}_h)^{-1}\|_{\L(E_0,E_k)}\leq \cE,\quad k=0,1,\quad \mu\in S_0. $$
Put 
$$\cB:=(\cB_j)_{j=0}^n \in \L(\mathring{B}^{1,\lambda^\prime-\lambda/2}_2(\M,V), E_0).$$
From Definition~\eqref{S2: Besov space}, it is not hard to show that
$$\mathring{\boldsymbol{B}}^{1,\lambda^\prime-\lambda/2}_2(\M,V)\doteq (E_1,E_0)_{1/2,2}. $$
Then by \eqref{S5.3: Besov-Lp}, we have
$$\cB\Lambda \in \L(\mathring{\boldsymbol{B}}^{1,\lambda^\prime-\lambda/2}_2(\M,V), E_0).$$
Combining with interpolation theory, we infer that 
for every $\varepsilon>0$ there exists some positive constant $C(\varepsilon)$ such that for all $\bu=(u_j)_{j=0}^n \in E_1$
$$
\|\cB\Lambda \bu\|_{E_0} \leq \varepsilon \|\bu\|_{E_1} + C(\varepsilon)\|\bu\|_{E_0}.
$$
Given any $\bu\in E_0$ and  $\mu\in S_0$,
\begin{align*}
\|\cB\Lambda (\mu +\bar{\cA}_h)^{-1}\bu\|_{E_0}
\leq & \varepsilon \|(\mu +\bar{\cA}_h)^{-1}\bu\|_{E_1} + C(\varepsilon) \|(\mu +\bar{\cA}_h)^{-1}\bu\|_{E_0}\\
\leq & \cE(\varepsilon   + \frac{C(\varepsilon)}{|\mu|}) \|\bu\|_{E_0}.
\end{align*}
Hence we can find some $\omega_1 \geq \omega_0$ such that for all $\mu\in S_1:=\omega_1 +\Sigma_\theta$
$$\|\cB\Lambda (\mu +\bar{\cA}_h)^{-1}\|_{\L(E_0)} \leq 1/2, $$
which implies that $S_1 \subset \rho(-\bar{\cA}_h-\cB\Lambda)$ and 
$$ \|(I +\cB\Lambda(\mu  +\bar{\cA}_h)^{-1})^{-1}\|_{\L(E_0)} \leq 2 .$$
Now one can easily verify that
$$|\mu|^{1-k} \| (\mu+ \bar{\cA}_h +\cB\Lambda)^{-1}\|_{\L(E_0,E_k)}\leq 2\cE,\quad k=0,1,\quad \mu\in S_1. $$

(iv) \eqref{S5: commutator} shows that
$$\Lambda^c (\mu+\Ah)u =(\mu+\bar{\cA}_h)\Lambda^c u +\cB\Lambda\Lambda^c u = (\mu +\bar{\cA}_h+\cB\Lambda)\Lambda^c u.$$
For any $\mu\in S_1$, this yields
$$\Lambda(\mu +\bar{\cA}_h+\cB\Lambda)^{-1}\Lambda^c (\mu+\Ah)=\Lambda(\mu +\bar{\cA}_h+\cB\Lambda)^{-1} (\mu +\bar{\cA}_h+\cB\Lambda)\Lambda^c=\id_{\mathring{W}^{2,\lambda^\prime-\lambda}_2(\M,V)}. $$
This proves the injectivity of $\mu+\cA$ for $\mu\in S_1$.

(v) On the other hand, one can also view $\cB_j$ as an operator from $\mathring{B}_2^{1,\lambda^\prime-\lambda/2}(X_j,V)$ to $L_2^{\lambda^\prime}(\M,V)$. Then
$$ \cB_j \in \L(\mathring{B}_2^{1,\lambda^\prime-\lambda/2}(X_j,V), L_2^{\lambda^\prime}(\M,V)).  $$
Let $\mathfrak{B}\bu:=\sum_{j=0}^n \cB_j u_j$ for $\bu=(u_j)_{j=0}^n$.
Following an analogous argument as in (iii), we infer that there exists some $\omega_2 \geq \omega_1$ such that  $S_2:=\omega_2 +\Sigma_\theta \subset \rho(-\bar{\cA}_h+\Lambda^c \fB)$
and
\begin{align}
\label{S5: eq1}
|\mu|^{1-k} \| (\mu+ \bar{\cA}_h -\Lambda^c \fB)^{-1}\|_{\L(E_0,E_k)}\leq 2\cE,\quad k=0,1,\quad \mu\in S_2.
\end{align} 
We further have
$$(\mu+\Ah)\Lambda(\mu+\bar{\cA}_h-\Lambda^c\fB)^{-1}\Lambda^c=\Lambda(\mu+\bar{\cA}_h-\Lambda^c\fB)(\mu+\bar{\cA}_h-\Lambda^c\fB)^{-1}\Lambda^c=\id_{L_2^{\lambda^\prime}(\M,V)} .$$
Thus, $\mu+\cA$ is surjective for $\mu\in S_2$. Moreover, together with \eqref{S5: eq1}, we have
$$|\mu|^{1-k} \| (\mu+ \Ah)^{-1}\|_{\L(L_2^{\lambda^\prime}(\M,V), \mathring{W}_2^{2k,\lambda^\prime-\lambda}(\M,V))}\leq \cE^\prime,\quad k=0,1,\quad \mu\in S_2$$
for some $\cE^\prime>0$.
Now the asserted statement follows from the well-known semigroup theory and a similar argument to the proof for Theorem~\ref{S5: A-semigroup}.
\end{proof}

The following theorem is the main result of this paper.
\begin{theorem}
\label{S5: main theorem}
Suppose that $(\M,g; \rho)$ is a singular manifold satisfying $\rho\leq 1$,
$$|\nabla \rho|_g \sim {\bf 1},\quad \|\Delta \rho\|_\infty<\infty$$
on $\M_r:=\{\p\in \M: \rho(\p)<r\}$ for some $r\in (0,1]$. Moreover, assume that the set
$$S_{r_0}:=\{\p\in \M: \rho(\p)=r_0\} $$
is compact for $r_0\in (0,r)$.
Let $\lambda^\prime\in \R$,  and $\lambda\in [0,1)\cup(1,\infty)$.  
\begin{itemize}
\item[(a)] Then $(\M,g;\rho)$ is a singular manifold with $\sH_\lambda$-ends. 
\item[(b)] Furthermore, assume that the differential operator
\begin{align*}
\cA u:= -\div(\rho^{2-\lambda} \gd u)+\ev(\nabla u, a_1) +a_0 u
\end{align*}
is $(\rho,\lambda)$-regular.
Then
$$\cA\in \cH( \mathring{W}^{2,\lambda^\prime -\lambda}_p(\M, V),  L^{\lambda^\prime}_p(\M, V)),\quad 1<p<\infty.$$
Here $V=\C$ if $p\neq 2$, or $V=V^\sigma_\tau$ with  $\sigma,\tau\in\Nz$ if $p=2$.
\end{itemize}
\end{theorem}
\begin{proof}
(i)
For $M>0$, we set 
\begin{align}
\label{S5.3: h}
h(\p)= M\sg(1-\lambda)\log \rho(\p), \quad \p\in\M. 
\end{align}
A direct computation shows that
$$\rho\gd h =M\sg(1-\lambda) \gd \rho,$$
and
$$
\rho^\lambda \div(\rho^{2-\lambda}\gd h)=M\sg(1-\lambda)\rho \Delta \rho + M|1-\lambda| |\gd \rho|_g^2.  $$
Together with (S3) and  (S4), one can then easily show that $h \in BC^{2,0}(\M)$,  and
\begin{align}
\label{S5.3: h on ends}
h\in \sH_\lambda(\M_{r_1},g; \rho) 
\end{align}
with parameters $(c,M)$ for some $r_1\leq r$ sufficiently small.
\smallskip\\
By the implicit function function theorem, $S_{r_0}$ is a compact submanifold. Then   the assertion that $(\M,g;\rho)$ is a {\em singular manifold} with $\sH_\lambda$-ends is simply a consequence of  the tubular neighborhood theorem.

(ii) 
The retraction-coretraction system defined in the proof for Theorem~\ref{S5: sing mfd-ends} allows us to decompose the problem into generation of analytic semigroup on every $\sH_\lambda$-end, and then to glue the complete operator together by the perturbation argument used therein. 

We thus can reduce the assumptions on the manifold $(\M,g;\rho)$ to only assuming $(\M,g;\rho)$ to be a {\em singular manifold} with {\em property $\sH_\lambda$}, and {\em property $\sH_0$} if $\lambda\neq 0$. Moreover, 
$$h= M\sg(1-\lambda)\log \rho \in \sH_\lambda(\M,g;\rho)  ,\quad h_0=  M\log \rho \in  \sH_0(\M,g;\rho),$$
both with parameter $(c,M)$.

The reason to include the extra assumption that $(\M,g;\rho)$ has {\em property $\sH_0$} will  be self-explanatory in Step (v) below, while we determine the domain of the $L_p^{\lambda^\prime}(\M)$-realization of the operator $\cA$.

(iii)
Take $h$ as in \eqref{S5.3: h} and $z=a+ib$, $M$ satisfying (H2) and (H3) in Section~5.1. In Theorem~\ref{S5: sing mfd-ends}, we have shown that 
$$\cA\in \cH(e^{zh}\mathring{W}^{2,\lambda^\prime -\lambda}_2(\M, V), e^{zh} L^{\lambda^\prime}_2(\M,V)).$$
We have $e^{zh}=\rho^{\sg(1-\lambda)zM}=\rho^{\sg(1-\lambda)aM} \rho^{\sg(1-\lambda) bMi}$.
By \eqref{S3: gd log r} and Proposition~\ref{S2: sharp-flat}, we infer that
$$\nabla \log \rho\in BC^{1,0}(\M,T^*M),$$
which implies 
$$\nabla \rho^{\sg(1-\lambda)bM i}= \sg(1-\lambda)bM i \rho^{\sg(1-\lambda)bM i} \nabla \log \rho \in   BC^{1,0}(\M,T^*M).$$
Combining with $|\rho^{\sg(1-\lambda)bM i}|\equiv 1$, we thus have 
$$\rho^{\sg(1-\lambda)bM i}\in BC^{2,0}(\M).$$
By Propositions~\ref{S2: pointwise multiplication}, \ref{S2: change of wgt} and the fact that
$e^{zh} e^{-zh}=e^{-zh} e^{zh} ={\bf 1}_{\M}$, we infer that
$$e^{zh}\in  \Lis(\mathring{W}^{s,\vartheta}_p(\M,V), \mathring{W}^{s,\vartheta+\sg(\lambda-1)aM}_p(\M,V)),\quad 1<p<\infty,\quad 0\leq s\leq 2.$$
A similar argument to Theorem~\ref{S5: A-semigroup} yields
$$\cA\in \cH(\mathring{W}^{2,\lambda^\prime -\lambda+\sg(\lambda-1)aM}_2(\M, V),  L^{\lambda^\prime+\sg(\lambda-1)aM}_2(\M,V)).$$
Since $\lambda^\prime$ is arbitrary and $\sg(\lambda-1)aM\in (-1/2c^3, 1/2c^3)$, it implies that 
$$\cA\in \cH(\mathring{W}^{2,\lambda^\prime -\lambda}_2(\M, V),  L^{\lambda^\prime}_2(\M,V)),\quad \lambda^\prime\in\R.$$

(iv)
Now we look at the general case $1<p<\infty$ and suppose that $V=\C$.
Recall that the adjoint, $\Ata$, of $\cA$ with respect to $L_2^{\vartheta/2}(\M)$ is
$$\Ata u = -\div(\rho^{2-\lambda} \gd u) -\ev(\nabla u , 2\vartheta\rho^{2-\lambda}\gd \log \rho+a_1 ) +  b(\vartheta,\ba) u, $$
where with $\ba=(\vec{a},a_1,a_0)$ and
\begin{align*}
b(\vartheta,\ba):= \bar{a}_0 - \div(\vartheta \rho^{2-\lambda}\gd\log\rho+a_1)
&-\vartheta(\vartheta \rho^{2-\lambda}\gd\log\rho+a_1 |  \gd\log\rho)_g.
\end{align*}
To simplify our usage of notation in the following computations, we first focus on the case $\lambda>1$. The remaining case follows easily by symmetry. Recall that when $\lambda>1$, we can set
$$h(\p)= -M\log \rho(\p), \quad \p\in\M,\quad M>0. $$
Let $\Aath:=e^{-zh}\circ \Ata \circ e^{zh}$. 
Since  $\Ata$is $(\rho,\lambda)$-regular and $(\rho,\lambda)$-singular elliptic, 
by choosing $z=z(\vartheta)=a+ib$ and $M=M(\vartheta)$ satisfying (H2) and (H3), we have
$$ \Aath \in \cH(\mathring{W}_2^{2,  -\lambda}(\M), L_2 (\M))\cap \Lis(\mathring{W}_2^{2,  -\lambda}(\M), L_2 (\M)). $$
We have thus established
$$ \Ata \in \cH(\mathring{W}_2^{2, aM-\lambda}(\M), L_2^{aM} (\M))\cap \Lis(\mathring{W}_2^{2, aM-\lambda}(\M), L_2^{aM} (\M)), $$
and the semigroup $\{e^{-t\Ata}\}_{t\geq 0}$ is contractive. 
Note that 
$$aM \in (-1/2c^3,0)\subset (-1,0)$$ 
only depends on $c$. Henceforth, we always take $\alpha:= aM=-1/4c^3$. 

For the adjoint, $\Azt$, of $\Ata$ with respect to $L_2^\alpha(\M)$, we can show similarly that
$$ \Azt \in \cH(\mathring{W}_2^{2, \alpha-\lambda}(\M), L_2^\alpha (\M))\cap \Lis(\mathring{W}_2^{2, \alpha-\lambda}(\M), L_2^\alpha (\M)), $$
and $\{e^{-t\Azt}\}_{t\geq 0}$ is contractive. Let $\cA_h(\vartheta;2\alpha)=e^{-zh}\circ \Azt \circ e^{zh}$.

The $L_\infty$-contractivity of $\{e^{-t\cA_h}\}_{t\geq 0}$, $\{e^{-t\Aath}\}_{t\geq 0}$ and  $\{e^{-t\cA_h(\vartheta;2\alpha)}\}_{t\geq 0}$ can be built up by a similar argument to Section~3.2. It yields for any $u\in L_2^\alpha(\M) \cap L_\infty^\alpha(\M)$
\begin{align}
\label{S5.3: A-Ah comm}
\notag \| e^{-t\cA} u \|_{L_\infty^\alpha} &=\|  e^{-t\cA}  e^{zh} e^{-zh}u\|_{L_\infty^\alpha}\\
&=\|e^{zh}   e^{-t\cA_h}  e^{-zh}u\|_{L_\infty^\alpha}\\
\notag&\leq \|e^{-t\cA_h}  e^{-zh}u\|_{L_\infty} \leq \|e^{-zh}u\|_{L_\infty} \leq \|u \|_{L_\infty^\alpha}.
\end{align}
\eqref{S5.3: A-Ah comm} follows from $\Ah= e^{-zh}\circ\cA \circ e^{zh}$ and 
\begin{align*}
e^{-zh} e^{-t\cA} e^{zh} v &= e^{-zh}\lim_{n\to \infty}[\frac{n}{t}( \frac{n}{t} +\cA)^{-1} ]^n e^{zh} v\\
&= \lim_{n\to \infty} e^{-zh} e^{zh} [\frac{n}{t} e^{-zh} ( \frac{n}{t} +\cA)^{-1} e^{zh}]^n v\\
&=\lim_{n\to \infty}[\frac{n}{t}  ( \frac{n}{t} +e^{-zh}\circ\cA \circ e^{zh})^{-1} ]^n v\\
&=\lim_{n\to \infty}[\frac{n}{t}  ( \frac{n}{t} +\Ah)^{-1} ]^n u = e^{-t\Ah} v.
\end{align*}
A similar argument applies to $\{e^{-t\Ata}\}_{t\geq 0}$ and $\{e^{-t\Azt}\}_{t\geq 0}$ as well. 
Thus we have established the $L_\infty^\alpha$-contractivity of the semigroups  $\{e^{-t\cA}\}_{t\geq 0}$, $\{e^{-t\Ata}\}_{t\geq 0}$, and $\{e^{-t\Azt}\}_{t\geq 0}$. 

Now we make use of the duality argument in Step~(ii) of the proof for Theorem~\ref{S3.2: Lp-continuity} again. For any $u\in L_2^{\vartheta-\alpha}(\M) \cap L_1^{\vartheta-\alpha}(\M)$ and  $v\in L_2^{\alpha}(\M) \cap L_\infty^{\alpha}(\M)$, it holds
\begin{align*}
|\langle e^{-t\cA}v | u \rangle_{2,{\vartheta/2}}| &= |\langle v | e^{-t\Aa}u \rangle_{2,{\vartheta/2}}|
=|\langle \rho^{\vartheta} v | e^{-t\Aa} u \rangle_{2,0}|\\
&\leq \|\rho^{\vartheta} v\|_{L_1^{-\alpha}} \|e^{-t\Aa}u \|_{L_\infty^{\alpha}}\\
&\leq \|v\|_{L_1^{\vartheta-\alpha}} \|   u \|_{L_\infty^{\alpha}} .
\end{align*}
Taking $\vartheta=\lambda^\prime+\alpha$, the above inequality proves that $\{e^{-t\cA}\}_{t\geq 0}$ is indeed $L_1^{\lambda^\prime}$-contractive. 
Applying this duality argument to $\{e^{-t\Ata}\}_{t\geq 0}$ and $\{e^{-t\Azt}\}_{t\geq 0}$ repeatedly with respect to $\langle \cdot | \cdot \rangle_{2,\alpha}$, we can then obtain the   $L_1^\alpha$-contractivity of these two semigroups. Adopting the duality argument once more, we have
\begin{align*}
|\langle e^{-t\cA}v | u \rangle_{2,{\vartheta/2}}| \leq \|v\|_{L_\infty^{\vartheta-\alpha}} \|   u \|_{L_1^{\alpha}} .
\end{align*}
Hence, by Lemma~\ref{S3.2: Riesz-Thorin} $\{e^{-t\cA}\}_{t\geq 0}$ is indeed $L_p^{\lambda^\prime}$-contractive for all $1\leq p \leq \infty$. After carefully following the proof for Theorem~\ref{S3.2: Lp-continuity}, one can show that $\{e^{-t\cA}\}_{t\geq 0}$ can be extended to a contractive strongly continuous analytic semigroup on $L^{\lambda^\prime}_p(\M)$ for $1<p<\infty$.

(v) To determine the domain of the realization of $\cA$ on $L^{\lambda^\prime}_p(\M)$, we look at the operator
$$\fB u:=-\div(\rho^{\lambda} \vec{a}\cdot \gd  u) +  \ev(\nabla u, \rho^{\lambda} a_1) +  \rho^{\lambda} a_0  u + \lambda\ev(\nabla u, \rho^\lambda\vec{a}\cdot \gd \log \rho). $$
We have computed in Section~3.1 that 
$\cA=  \rho^{-\lambda} \fB$.
Since in Step (ii), we assume that $(\M,g;\rho)$ has {\em property $\sH_0$}, 
following an analogous discussion to Step (iii)-(iv), we can show that $-\fB_\omega:=-\fB -\omega$ generates a contractive strongly continuous analytic semigroup on $L^{\lambda^\prime-\lambda}_p(\M)$ with domain $D(\fB_\omega)$ for $1<p<\infty$ and any $\omega\geq 0$. 
In particular, $D(\fB_\omega)$ is independent of $\omega$.
On the other hand, by \cite[Theorem~5.2]{Ama13b}, for $\omega$ sufficiently large and $\lambda^\prime\in\R$, $1<p<\infty$
$$\fB_\omega\in \cH( \mathring{W}^{2,\lambda^\prime -\lambda}_p(\M),  L^{\lambda^\prime-\lambda}_p(\M)) \cap \Lis( \mathring{W}^{2,\lambda^\prime -\lambda}_p(\M),  L^{\lambda^\prime-\lambda}_p(\M)).$$
Therefore, we indeed have $D(\fB)\doteq \mathring{W}^{\lambda^\prime -\lambda}_p(\M)$ and 
$$\fB\in \cH( \mathring{W}^{2,\lambda^\prime -\lambda}_p(\M),  L^{\lambda^\prime-\lambda}_p(\M)) \cap \Lis( \mathring{W}^{2,\lambda^\prime -\lambda}_p(\M),  L^{\lambda^\prime-\lambda}_p(\M)).$$
Now it follows from a similar argument to the proof for Theorem~\ref{S3: domain} that
$$\cA\in \cH( \mathring{W}^{2,\lambda^\prime -\lambda}_p(\M),  L^{\lambda^\prime}_p(\M)) \cap \Lis( \mathring{W}^{2,\lambda^\prime -\lambda}_p(\M),  L^{\lambda^\prime}_p(\M)),\quad \lambda^\prime\in\R, \quad 1<p<\infty.$$ 
\end{proof}

We say $u,v \in \R^\M$ are $C^k$-equivalent, which is denoted by $u\sim_k v$, if
$$u\sim v,\quad |\nabla ^i u |_g \sim | \nabla^i v|_g,\quad i=1,\cdots,k. $$
\begin{definition}
An $m$-dimensional  singular manifold $(\M,g;\rho)$ is called a singular manifold with holes and uniformly mild wedge ends if it fulfils the following conditions.
\begin{itemize}
\item[(i)] $(\sM,g)$ is an $m$-dimensional uniformly regular Riemannian manifold, and ${\bf \Sigma}=\{\Sigma_1,\cdots,\Sigma_k\}$ is a finite set of disjoint $m$-dimensional compact manifolds with boundary such that $\Sigma_j\subset \mathring{\sM}$. Put $G_0:=\sM\setminus \cup_{j=1}^k \Sigma_j$ and 
$$\mathscr{B}_{j,r}:= \bar{\B}_\sM(\partial\Sigma_j,r)\cap G_0,\quad j=1,\cdots,k.$$ 
Furthermore, the singularity function $\rho$ satisfies 
\begin{align}
\label{S5.3: near bdry}
\rho \sim_2 {\rm dist}_{\sM}(\cdot, \partial \Sigma_j)\quad \text{in } \mathscr{B}_{j,r} 
\end{align}
for some $r\in (0,\delta)$, where $\delta< {\rm diam}(\sM)$ and $\mathscr{B}_{i,\delta} \cap \mathscr{B}_{j,\delta}=\emptyset$ for $i\neq j$, and
$$\rho\sim {\bf 1},\quad \text{elsewhere on }G_0. $$
\item[(ii)] ${\bf G}=\{G_1,\cdots,G_n\}$ is a finite set of disjoint $m$-dimensional uniformly mild wedges.
More precisely, there is a diffeomorphism $f_i: G_i \to W(R_i,B_i,\Gamma_i)$ with $R_i\in \mathscr{C_U}(J_0)$.
Let $I_r:=(0,r]$ and
$$\mathscr{G}_{i,r}:=f_i^{-1} (\phi_P(I_r\times B_i)\times\Gamma_i),\quad i=1,\cdots,n.$$
Moreover, the singularity function $\rho$ satisfies 
\begin{align}
\label{S5.3: on wedge end}
\rho \sim_2 f_i^*(\phi_P^*(R_i|_{I_r} \otimes {\bf 1}_{B_i})\otimes {\bf 1}_{\Gamma_i}) \quad \text{in } \mathscr{G}_{j,r}
\end{align}
for some $r\in (0,1]$, and
$$\rho\sim {\bf 1},\quad \text{elsewhere on }G_i. $$
\item[(iii)] $\{G_0\}\cup {\bf G}$ forms a covering for $\M$. $\partial_0 G_i:= G_0\cap G_i \subset \partial G_0 \cap \partial G_i$.
\end{itemize}
\end{definition}
One can easily see that \eqref{S5.3: near bdry} and \eqref{S5.3: on wedge end} imply that
\begin{align}
\label{S5.3: equiv near ends}
|\Delta \rho|<\infty \quad \text{in }\mathscr{B}_{j,r} \text{ and } \mathscr{G}_{j,r}.
\end{align}


The following corollary does not directly stem from Theorems~\ref{S5: sing mfd-ends} and \ref{S5: main theorem}. But using the ideas in their proofs, we can prove this corollary without difficulty.
\begin{cor}
\label{S5: main corollary}
Suppose that $(\M,g; \rho)$ is a singular manifold with holes and uniformly mild wedge ends. 
Let $\lambda^\prime\in \R$,  and $\lambda\in [0,1)\cup(1,\infty)$. 
Furthermore, assume that the differential operator
\begin{align*}
\cA u:= -\div(\rho^{2-\lambda} \gd u)+\ev(\nabla u, a_1) +a_0 u
\end{align*}
is $(\rho,\lambda)$-regular.
Then
$$\cA\in \cH( \mathring{W}^{2,\lambda^\prime -\lambda}_p(\M, V),  L^{\lambda^\prime}_p(\M, V)),\quad 1<p<\infty.$$
Here $V=\C$ if $p\neq 2$, or $V=V^\sigma_\tau$ with $\sigma,\tau\in\Nz$ if $p=2$.
\end{cor}
\begin{proof}
If $S_{i,r}:=\{\p\in G_i: \rho(\p)=r\}$ is compact for small $r$ and all $i=1,\cdots,n$, then by Theorem~\ref{S5: main theorem} the asserted result will be true. However, in general, $S_{i,r}$ might not be compact. 
Nevertheless, looking into the proofs for Theorem~\ref{S5: sing mfd-ends} and Theorem~\ref{S5: main theorem}, the compactness of $S_{i,r}$ will only be responsible for Step~(i) and (ii) in the proof for Theorem~\ref{S5: sing mfd-ends}.

Firstly, we take
$h:=\sg (1-\lambda)\log \rho. $
Then $h\in C^2(\M)$ satisfies
$$
h\in \sH_\lambda(\mathscr{B}_{j,r},g; \rho)\, \text{ and } \, h\in \sH_\lambda(\mathscr{G}_{i,r},g; \rho) 
$$
with parameters $(c,1)$ for some $r>0$, following from \eqref{S5.3: equiv near ends} and a similar argument to the proofs for Propositions~\ref{S5: Eg-mild cusp-wedge} and \ref{S5: sing mdf with holes}. 
Furthermore, 
$$
h\in BC^2(\M\setminus (\bigcup_{j=1}^k \mathring{\mathscr{B}}_{j,r}\cup \bigcup_{i=1}^n \mathring{\mathscr{G}}_{i,r})).
$$
Thus the properties of $h$ listed in Step~(ii) of the proof for Theorem~\ref{S5: sing mfd-ends} are all satisfied.

Next, we prove the existence of the retraction-coretraction system defined in Step~(i) of  the proof for Theorem~\ref{S5: sing mfd-ends}. 
For any $r\in (0,1)$,
picking $(r_0,r]$ with $r_0>0$, we can construct a collar neighborhood of $S_{i,r}$ on $G_i$ by
$$U_i:=f_i^{-1}(\phi_P^{-1}(B_i\times (r_0,r])\times \Gamma_i).$$
Moreover, we choose $\xi, \tilde{\xi} \in BC^\infty((r_0,r], [0,1])$ such that
$$\xi|_{[\frac{r+r_0}{2}, r]}\equiv 1,\quad \xi|_{(r_0, \frac{r+3r_0}{4}]}\equiv 0;\quad  \tilde{\xi}|_{(r_0,\frac{r+r_0}{2}]}\equiv 1,\quad \tilde{\xi}|_{[\frac{3r+r_0}{4},r]}\equiv 0.$$ 
Now we can define 
$\hat{\pi}_{i,0}:= f_i^*( \phi_P^*({\bf 1}_{B_i} \otimes \xi)\otimes {\bf 1}_{\Gamma_i}), $
and 
$\hat{\pi}_{i,i}:= f_i^*( \phi_P^*({\bf 1}_{B_i} \otimes \tilde{\xi})\otimes {\bf 1}_{\Gamma_i}). $
The rest of the proof just follows from a similar argument to Step~(i) of the proof for Theorem~\ref{S5: sing mfd-ends}.
\end{proof}
\begin{remark}
\label{S5.3: rmk-remove bdry/pts}
In view of Remarks~\ref{S5.2: remove bdry} and \ref{S5.2: remove pts}, the assertion in Corollary~\ref{S5: main corollary} remains true if we replace the condition of {\em singular manifolds with holes} by removing a finite set of disjoint compact submanifolds $\{\Sigma_1,\cdots,\Sigma_k\}$ or discrete points $\{\p_1,\cdots,\p_k\}$ from a {\em uniformly regular Riemannian manifold} $(\mathscr{M},g)$. Here $\Sigma_i\subset \partial\mathscr{M}$ if $\Sigma_i\cap \partial\mathscr{M}\neq \emptyset,$ or $\p_i \in \mathring{\mathscr{M}}$.
\end{remark}

\begin{remark}
\label{S5.3: main remark}
From our proofs in  Section 3 and 5, it is a simple matter to check that we do not require the {\em singular manifold} $(\M,g;\rho)$ to enjoy smoothness up to $C^\infty$. Indeed, in order to prove all the results in Section 3 and 5, it suffices to require $(\M,g;\rho)$ to be a $C^2$-{\em singular manifold}.
\end{remark}


\section{\bf Applications}

\subsection{\bf The Laplace-Beltrami operator}
Suppose that $(\M,g; \rho)$ is a {\em singular manifold}.

Recall that the Laplace-Beltrami operator with respect to $g$ is defined by 
$$\Delta=\Delta_g:=\div\circ\gd. $$
One readily checks that $\Delta$ is $(\rho,\lambda)$-regular and $(\rho,\lambda)$-singular elliptic with $C_{\hat{\sigma}}=1$, $\lambda=2$.

Given any Banach space $X$, $s\in (0,1)$, and any perfect interval $J$, we denote by 
$$C^s(J,X) $$
the set of all $u\in C(J,X)$ such that $u$ is H\"older continuous of order $s$.

Let $\M_T:=(0,T]\times\M$, and $\M_0:=\{0\}\times\M$.
Then Theorem~\ref{S5: main theorem}, Corollary~\ref{S5: main corollary} and \cite[Theorem~II.1.2.1]{Ama95}  imply the following existence and uniqueness theorem for the  heat equation.
\begin{theorem}
\label{S6.1: heat-thm}
Suppose that either $(\M,g; \rho)$ is a   singular manifold with holes and uniformly mild wedge ends, or  $(\M,g; \rho)$  satisfies the conditions in Theorem~\ref{S5: main theorem}. Let $\lambda^\prime\in \R$ and $J=[0,T]$. 
Then for any 
$$(f,u_0) \in C^s(J;L^{\lambda^\prime}_p(\M))\times L^{\lambda^\prime}_p(\M), $$ 
with some $s\in (0,1)$,
the boundary value problem
\begin{align}
\label{S6.1: heat eq}
\left\{\begin{aligned}
u_t -\Delta u&= f &&\text{on}&&\M_T\\
u&=0 &&\text{on}&& \partial\M_T &&\\
u&=u_0 &&\text{on}&& \M_0 &&
\end{aligned}\right.
\end{align}
has a unique solution 
$$u\in C^{1+s}(J\setminus \{0\};L^{\lambda^\prime}_p(\M))\cap C^s(J\setminus \{0\};\mathring{W}^{\lambda^\prime-2}_p(\M)).$$
\end{theorem}
In the case of $L_2$-spaces, making use of \cite[Theorem~1.6]{Pru03}, we have the following corollary.
\begin{cor}
Under the conditions in Theorem~\ref{S6.1: heat-thm}, 
let $V=V^\sigma_\tau$ be a tensor field on $\M$ and $1<p<\infty$.
Then for any 
$$(f,u_0) \in L_p([0,T];L^{\lambda^\prime}_2(\M, V))\times \mathring{B}^{1,\lambda^\prime-1}_2(\M, V), $$ 
the boundary value problem~\eqref{S6.1: heat eq}
has a unique solution 
$$u\in   L_p([0,T];\mathring{W}^{,\lambda^\prime-2}_2(\M, V))\times W^1_p([0,T];L^{\lambda^\prime}_2(\M, V)).$$
\end{cor}
\begin{remark}
A similar result can also be formulated for the wave equation on {\em singular manifolds with holes and uniformly mild wedge ends}, or {\em singular manifolds}  satisfying the conditions in Theorem~\ref{S5: main theorem}. We refer the reader to \cite{Pazy83} for the corresponding semigroup theory for hyperbolic equations.
\end{remark}

\subsection{\bf Degenerate and singular equations on domains}
Suppose that $\Omega\subset \R^m$ is a smooth domain with compact boundary.
 
For $r$ small enough, $\partial\Omega$ admits an $r$-tubular neighborhood, which we denote by ${\sf T}_r$. Here $r$ depends on the {\em uniform exterior and interior ball condition} of $\partial\Omega$.
Let 
$$d_{\partial\Omega}(x):={\sf dist}(x, \partial\Omega),\quad x\in \Omega,$$ 
i.e., the distance function to the boundary.
We define $\d: \Omega\rightarrow \R^+$ by 
\begin{equation}
\label{S6.2: rescaled dist}
\d=d_{\partial\Omega} \quad \text{if $\Omega$ is bounded},\quad \text{or} \quad 
\begin{cases}
\d=d_{\partial\Omega} \quad &\text{in } \Omega\cap {\sf T}_r,\\
\d\sim {\bf 1} &\text{in }\Omega\setminus {\sf T}_r
\end{cases}
\quad \text{otherwise}.
\end{equation}
It follows from \cite[formula~(3.3), Corollary~4.5]{Ama14} that $\bar{\Omega}$ is a {\em uniformly regular Riemannian manifold}.
Then by Remarks~\ref{S5.2: remove bdry}, \ref{S5.3: rmk-remove bdry/pts}, and Corollary~\ref{S5: mfd with cpt bdry},
$(\Omega, g_m; \d)$ is a {\em singular manifold} with $\sH_\lambda$-ends for $\lambda\in [0,1)\cup (1,\infty)$. The weighted Sobolev-Slobodeckii spaces and H\"older spaces can thus be defined as in Section~2.2, which is denoted by $\F^{s,\vartheta}(\Omega,X)$ for any finite dimensional Banach space $X$.


Let $J=[0,T]$. We consider the following initial value problem.
\begin{align}
\label{S6.2: degenerate BVP}
\left\{\begin{aligned}
u_t + \cA u &= f &&\text{on}&&\Omega_T;\\
u&=u_0 &&\text{on}&& \Omega_0 .&&
\end{aligned}\right.
\end{align}
Here 
$$ \cA u=-a \Delta u + a_1 \cdot \nabla u + a_0 u, $$
and the coefficients $ (a,a_1,a_0)$ satisfy for some $s\in (0,1)$ and $\lambda\in [0,1)\cup (1,\infty)$
\begin{align}
\label{S6.2: lower order coef}
 a_1\in C^s(J; BC^{0,\lambda}(\Omega, T_0\R^m)),\quad  a_0\in C^s(J;L_\infty^\lambda(\Omega));
\end{align}
and if $\lambda=2$
\begin{align}
\label{S6.2: prc part-1}
a\in C^s(J;\R_+) ;
\end{align}
or if $\lambda\neq 2$
\begin{align}
\label{S6.2: prc part-2}
a\in C^s(J, BC^{2,\lambda-2}(\Omega)),\quad \text{ for every}\quad  t\in J,\quad  a(t)^{\frac{1}{2-\lambda}}\sim_2 \d.
\end{align}
Observe that \eqref{S6.2: lower order coef} can be equivalently stated as
$$\d^{\lambda-1} a_1\in C^s(J; BC(\Omega, \R^m)),\quad \d^{\lambda} a_0\in C^s(J;L_\infty(\Omega)).$$
By \eqref{S6.2: prc part-1} and \eqref{S6.2: prc part-2}, we can verify that when $\lambda\neq 2$, $(\Omega, g_m; a^{\frac{1}{2-\lambda}})$ is a $C^2$-{\em singular manifold} with $\sH_\lambda$-ends. 
When $\lambda=2$, we take the {\em singular manifold} to be $(\Omega, g_m, \d)$. 
In both cases, the conditions in Theorem~\ref{S5: main theorem} are satisfied.

Now we conclude from \cite[Theorem~II.1.2.1]{Ama95} that
\begin{theorem}
Suppose that $\Omega\subset \R^m$ is a smooth domain with compact boundary. 
Let $s\in (0,1)$, $\lambda\in [0,1)\cup (1,\infty)$, $\lambda^\prime\in\R$ and $1<p<\infty$. 
Assume that the coefficients $(a,a_1,a_0)$ of the differential operator  
$$ \cA u=-a \Delta u + a_1 \cdot \nabla u + a_0 u $$
satisfy \eqref{S6.2: lower order coef}-\eqref{S6.2: prc part-2}. Then given any 
$$(f,u_0) \in C^s(J; L_p^{\lambda^\prime}(\Omega))\times L_p^{\lambda^\prime}(\Omega),$$
the initial value problem~\eqref{S6.2: degenerate BVP} has a unique solution 
$$u\in C^{1+s}(J\setminus \{0\};L^{\lambda^\prime}_p(\Omega))\cap C^s(J\setminus \{0\};W^{\lambda^\prime-\lambda}_p(\Omega)).$$
\end{theorem}
\begin{remark}
Based on \eqref{S6.2: prc part-2}, we can readily observe that the principle symbol of $\cA$ satisfies
$$ a(t)|\xi|^2 \sim \d^{2-\lambda} |\xi|^2, \quad \lambda\neq 2. $$
Therefore, \eqref{S6.2: degenerate BVP} can either be a degenerate boundary value problem or be a boundary blow-up problem.
This supplements the results in \cite{ForMetPall11, Shao14, Vesp89} with weak degeneration case, i.e., $\lambda\in (0,2)$, and boundary singularity case, i.e., $\lambda>2$.
\end{remark}

\subsection{\bf Generalized Heston operator}
Let $\Omega=\R\times \R_+$.
One can readily check that 
$$(\M,g;\rho):=(\Omega , g_2; y),\quad g_2=dx^2+dy^2,$$
is a {\em singular manifold} with uniformly mild wedge end.

Let $J:=[0,T]$. Consider the following initial value problem. 
\begin{align}
\label{S6.3: gen-Heston-1}
\left\{\begin{aligned}
u_t +\cA u&= f &&\text{on}&&\Omega_T\\
u(0)&=u_0  &&\text{on}&&\Omega_0. &&
\end{aligned}\right.
\end{align}
Here with $\alpha\leq 2$ and $z=(x,y)$ 
\begin{align*}
&\quad \cA(t,z)u(t,z)\\
:&= -\partial_i( y^\alpha a^{ij}\partial_j u(t,z)) + y^{\alpha-1} b^j(t,z)\partial_j u(t,z) +y^{\alpha-2} c(t,z)u(t,z),
\end{align*}
where $b^j(t,z):=b^j_0(t,z) +y b^j_1(t,z)$, and $c(t,z):=c_0(t,z) +y c_1(t,x)+ y^2 c_2(t,z)$. We impose the following assumptions on the coefficients.
\begin{itemize}
\item[(GH1)] 
\begin{align*}
(a^{ij})=\frac{1}{2}
\begin{pmatrix}
1 &\varrho\sigma\\
\varrho\sigma &\sigma^2
\end{pmatrix}, \quad
\sigma>0,\quad -1<\varrho<1,
\end{align*}
\item[(GH2)] $b_i^j, c_i \in C^s(J;L_\infty(\Omega)) $ for some $s\in (0,1)$.
\end{itemize}
This problem corresponds to the case $\lambda=2-\alpha$ in \eqref{S3: operator A}.

While $\alpha=1$, $b_i^j\equiv \text{const}$, $c_0\equiv 0$, $c_1, c_2\equiv \text{const}$,
$\cA$ is called the Heston operator. 
\eqref{S6.3: gen-Heston-1} generalizes the Heston model in the following sense. 
It does not only exhibit degeneracy along the boundary, but boundary singularities may also appear. When $\alpha>0$, the diffusion term is degenerate. Whereas $\alpha<0$ corresponds to the situation that boundary singularities show for the highest order term. 

The Heston operator has been  studied in \cite{Fee13,FeePop13,FeePop14} and the references therein.  In this subsection, we focus on the case $\alpha\neq 1$. The study of this kind of problem is new since the Schauder approach in the aforementioned articles relies on the particular choice the degeneracy factor $y$.

One can check by direct computations that after a change of spatial variables and rescaling of the temporal variable. Equation~\eqref{S6.3: gen-Heston-1} can be transformed into 
\begin{align*}
\left\{\begin{aligned}
u_t +\hat{\cA} u&= f &&\text{on}&&\Omega_T\\
u(0)&=u_0  &&\text{on}&&\Omega_0. &&
\end{aligned}\right.
\end{align*}
Here 
\begin{align*}
\hat{\cA}(t,z)u(t,z):=& -\partial_j( y^\alpha \partial_j u(t,z)) + y^{\alpha-1} (\hat{b}^j_0(t,z) +y \hat{b}^j_1(t,z))\partial_j u(t,z) \\
&+y^{\alpha-2} (\hat{c}_0(t,z) +y \hat{c}_1(t,x)+ y^2 \hat{c}_2(t,z))u(t,z),
\end{align*}
where $\hat{b}_i^j, \hat{c}_i \in C^s(J;L_\infty(\Omega)) $. By Corollary~\ref{S5: main corollary}, 
\begin{theorem}
Suppose that  $\Omega=\R\times \R_+$.
Let $s\in (0,1)$, $\alpha\in (-\infty,1)\cup(1,2]$, $\lambda^\prime\in\R$ and $1<p<\infty$. 
Assume that (GH1) and (GH2) are satisfied. Then given any 
$$(f,u_0) \in C^s(J; L_p^{\lambda^\prime}(\Omega))\times L_p^{\lambda^\prime}(\Omega),$$
the equation \eqref{S6.3: gen-Heston-1} has a unique solution
$$u\in C^{1+s}(J\setminus \{0\};L^{\lambda^\prime}_p(\Omega))\cap C^s(J\setminus \{0\};W^{\lambda^\prime+\alpha-2}_p(\Omega)).$$
\end{theorem}

\section*{Acknowledgements}
The author would like to express his sincere gratitude to Prof. Herbert Amann and  Prof. Gieri Simonett for  valuable suggestions and helpful discussions.

\end{document}